\RequirePackage{fix-cm}
\documentclass[smallextended]{svjour3}       
\smartqed  

\usepackage{graphicx}
\usepackage{subfigure}
\usepackage{amssymb}
\usepackage{mathrsfs}
\usepackage{amsmath}
\usepackage{amsfonts}
\usepackage{epstopdf}
\usepackage{epsfig}
\usepackage{version,fancybox,amscd}
\usepackage{cases,bm,multirow}
\usepackage{bm}
\usepackage{float}
\usepackage{color}
\usepackage{geometry}
\usepackage{algorithm}
\usepackage{algorithmicx}
\usepackage{algpseudocode}
\def \div {\mathrm{div}}

\newtheorem{thm}{\bf Theorem}[section]
\newtheorem{lem}[thm]{\bf Lemma}
\newtheorem{den}[thm]{\bf Definition}

\begin{document}

\title{Spatially Adapted First and Second Order Regularization for Image Reconstruction: From an Image Surface Perspective}

\titlerunning{SA-TV-TV$^2$ Regularization for Image Reconstruction}        

\author{Qiuxiang Zhong  \and
        Ryan Wen Liu    \and
        Yuping Duan$^\ast$   
}

\institute{Q. Zhong \at
              Center for Applied Mathematics, Tianjin University, China
           \and
           R. Liu \at
              School of Navigation, Wuhan University of Technology, China
           \and
              Y. Duan$^\ast$ \at
              Center for Applied Mathematics, Tianjin University, China\\
              \email{yuping.duan@tju.edu.cn}
}

\date{Received: date / Accepted: date}

\maketitle

\begin{abstract}
In this paper, we propose a new variational model for image reconstruction by minimizing the $L^1$ norm of the \emph{Weingarten map} of image surface $(x,y,f(x,y))$ for a given image $f:{\mathrm{\Omega}}\rightarrow \mathbb R$. We analytically prove that the Weingarten map minimization model can not only keep the greyscale intensity contrasts of images, but also preserve edges and corners of objects. The alternating direction method of multiplier (ADMM) based algorithm is developed, where one subproblem needs to be solved by gradient descent.
In what follows, we derive a hybrid nonlinear first and second order regularization from the Weingarten map, and present an efficient ADMM-based algorithm by regarding the nonlinear weights as known. By comparing with several state-of-the-art methods on synthetic and real image reconstruction problems, it confirms that the proposed models can well preserve image contrasts and features, especially the spatially adapted first and second order regularization economizing much computational cost.

\end{abstract}

\section{Introduction}
\label{intro}

Image restoration has attracted extensive attention in the fields of image processing and computer vision, where variational formulations are particularly effective in high-quality recovery. Let ${\mathrm{\Omega}}$ be an open bounded subset of $\mathbb R^{n}$ with Lipschitz continuous boundary,  $f:{\mathrm{\Omega}}\rightarrow \mathbb{R}$ be a given image defined on the domain ${\mathrm{\Omega}}$, and $u:{\mathrm{\Omega}}\rightarrow \mathbb{R}$ be the latent clean image. Rudin, Osher and Fatemi \cite{rudin1992nonlinear} proposed the total variation (TV) regularization for image restoration as the following constrained minimization problem
\begin{equation}
\min_{u} ~ \int_{\mathrm{\Omega}}|\nabla u| dx,\quad
\mbox{with}~\int_{\mathrm{\Omega}} udx = \int_{\mathrm{\Omega}} fdx~\mbox{and}~\int_{\mathrm{\Omega}}(u-f)^2 = \sigma^2,
\label{conrof}
\end{equation}
where the constraints correspond to the assumption that the noise is of zero mean and standard deviation $\sigma$, and $|\cdot|$ denotes the Euclidean norm of the gradient vector for each pixel $x\in \rm\Omega$. Chambolle and Lions \cite{chambolle1997image} linked the constrained minimization problem \eqref{conrof} and the following minimization problem
\begin{equation}
\min_{u}~\int_{\mathrm{\Omega}}|\nabla u|dx + \frac{1}{2\lambda}\int_{\mathrm{\Omega}}(u-f)^2dx,
\label{ROF}
\end{equation}
where $\lambda\equiv \mbox{const}>0$ represents the Lagrange multiplier associated with the constraints. Indeed, an alternative way to express the TV model \eqref{ROF} for image reconstruction is given as
\begin{equation}
\min_{u}~\alpha\int_{\mathrm{\Omega}}|\nabla u|dx+\frac12\int_{\mathrm{\Omega}}(u-f)^2dx,
\label{wROF}
\end{equation}
where $\alpha=\lambda>0$ is the regularization parameter. The Lagrange multiplier $\lambda$ in \eqref{ROF} and the regularization parameter $\alpha$ in \eqref{wROF} are used to control the trade-off between the data fidelity and regularization, the best value of which can be estimated by Morozov's discrepancy principle  \cite{wen2012parameter} or in a bilevel optimization framework \cite{kunisch2013bilevel}. Although the TV regularization can help to remove the noises and preserve sharp edges, it also possesses some unfavorable properties to $u$, e.g., staircase effect and contrast reduction \cite{Meyer2001Oscillating,strong2003edge,zhu2012image}.

Because images are comprised of multiple objects at different scales, it is more reasonable to use spatially varying variables instead of constant values. Bertalm{\'\i}o \emph{et al.} \cite{bertalmio2003tv} proposed a variant TV restoration model using a set of $\{\lambda_i\}_{i=1}^r$ with each one corresponding to a region set $\{{\mathrm{\Omega}}_i\}_{i=1}^r$ of the image, where $\{{\mathrm{\Omega}}_i\}_{i=1}^r$ can be obtained by simple segmentation algorithms.
Almansa \emph{et al.} \cite{almansa2008tv} further developed the idea in \cite{bertalmio2003tv} by using local variance estimation for obtaining $\lambda(x):{\mathrm{\Omega}}\rightarrow\mathbb R$ without involving the segmentation in the process.
Gilboa \emph{et al.} \cite{gilboa2006variational} designed a pyramidal structure-texture decomposition of images, which isolated the noise and then estimated the spatially varying constraints based on local variance measures.
Dong \emph{et al.} \cite{dong2011automated} improved the local variance estimator for $\lambda(x)$ and update it automatically in a multi-scale TV scheme for removing Gaussian-distributed noise. Chung \emph{et al.} \cite{chung2016learning} used a bilevel optimization approach in function space for the choice of spatially dependent regularization parameter for \eqref{ROF}.
In the case of impulsive noise, Hinterm{\"u}ller and Rincon-Camacho \cite{hintermuller2010expected} proposed to develop the TVL1 model with spatially adapted regularization parameters based on local expected absolute value estimation for enhancing the image details and preserving the image edge.
Another branch of these methods pursues a spatially varying $\alpha(x):{\mathrm{\Omega}}\rightarrow\mathbb R$ for \eqref{wROF}, which are also known as weighted TV.
Strong and Chan \cite{strong2003edge} considered $\alpha(x)$ as a spatially adapted weight in TV regularization to remove smaller-scaled noise while leaving lager-scaled features essentially intact.
Yuan \emph{et al.} \cite{yuan2012multiframe} proposed a spatially weighted TV model in multi-frame super-resolution reconstruction for efficiently reducing the staircase effect and preserving the edge information.
Langer \cite{langer2017automated} realized the automated parameter selection of \eqref{wROF} based on the discrepancy principle.
Recently, Hinterm{\"u}ller \emph{et al.} \cite{hintermuller2017optimal1,hintermuller2017optimal2} computed the spatially adaptive weights for \eqref{wROF} using a bilevel optimization approach.

Although spatially varying $\lambda(x)$ or $\alpha(x)$ in the Rudin-Osher-Fetami model \eqref{ROF} and \eqref{wROF} can improve the reconstruction quality, they can not eliminate the staircase effect in the relatively large piecewise linear regions. Thus, high order variational models are proposed and studied in the last two decades.
Lysaker \emph{et al.} \cite{lysaker2003noise,lysaker2006iterative} proposed the noise removal model using the high order regularization term, that is
\begin{equation}
\min_{u}~\int_{\mathrm{\Omega}} |\nabla^2u|_{ F} dx + \frac{1}{2\lambda}\int_{\mathrm{\Omega}}(u-f)^2dx,
\label{htv}
\end{equation}
where $\nabla^2 u$ is the Hessian of $u$ and $|\nabla^2 u|_F=\sqrt{|u_{xx}|^2 + |u_{xy}|^2+ |u_{yx}|^2+|u_{yy}|^2}$ is the Frobenius norm defined on each pixel $x\in\rm\Omega$. The optimality condition of \eqref{htv} gives a fourth-order partial differential equation, which has been further studied both theoretically and numerically in \cite{hinterberger2006variational,chan2007image,wu2010augmented,papafitsoros2014combined}.
Papafitsoros and Sch{\"o}nlieb \cite{papafitsoros2014combined} suggested the following combined first and second order variational model
\begin{equation}
\min_{u}~\alpha \int_{\mathrm{\Omega}}|\nabla u|dx+\beta\int_{\mathrm{\Omega}}|{\nabla}^2 u|_Fdx+\frac{1}{2}\int_{\mathrm{\Omega}}(u-f)^2dx,
\label{PS}
\end{equation}
where $\alpha$ and $\beta$ are positive constants. The idea of the model \eqref{PS} is to regularize the reconstructed image with a fairly large weight $\alpha$ in the first order term to preserve the jumps and a not too large weight $\beta$ for the second order term to eliminate the staircase effect without introducing any serious blur.
Another important high order TV model was proposed by Bredies \emph{et al.} \cite{bredies2010total}, the so-called total generalized variation (TGV), which can integrate to incorporate smoothness from the first up to the $k$-th derivatives.

In addition, geometric attributes of curves and surfaces also provide high order regularization for image processing tasks. The well-known Euler's elastica model \cite{shen2003euler,tai2011fast,yashtini2016fast,Deng2019} minimizes the total elastica of all level curves in images, which reads
\begin{align}
\min_{u} ~ \int_\mathrm\Omega \bigg(a+b\Big(\nabla\cdot\frac{\nabla u}{|\nabla u|}\Big)^2\bigg)|\nabla u|dx+\frac{1}{2\lambda}\int_\mathrm\Omega(u-f)^2dx.
\label{euler-elastica}
\end{align}
Due to the strong priors for the continuity of edges provided by Euler's elastica, it has been used as the regularization for various shape and image processing tasks \cite{Chambolle2019}, such as image inpainting, shape completion, and shown to be able to achieve better restoration results than the TV regularization.

By considering the associated \emph{image surface} or \emph{graph} of $f$ in $\mathbb R^{n+1}$, the noise removal problem becomes the task of finding an approximate piecewise smooth surface \cite{Lysaker2004}. Then it is straightforward to employ the geometric invariants, e.g., \emph{mean curvature} and \emph{Gaussian curvature}, as the regularization term for image surface processing. Zhu and Chan \cite{zhu2012image} proposed the following mean curvature minimization model for image denoising
\begin{equation}\label{mcdenoising}
\min_{u}~\int_{\mathrm{\Omega}}\bigg|\nabla\cdot\frac{\nabla u}{\sqrt{1+|\nabla u|^2}}\bigg| dx + \frac{1}{2\lambda}\int_{\mathrm{\Omega}}(u-f)^2dx,
\end{equation}
where $|\cdot|$ is actually the absolute value norm, also equivalent to Euclidean norm of one-dimensional vectors.
The $L^1$ norm of mean curvature is shown to be a desirable regularization for image denoising, which can not only preserve image contrast and corners of objects, but also remove the staircase effect. The Gaussian curvature has also been used as the regularization term for image denoising problems \cite{brito2016image}
\begin{equation}\label{gauss-curvature-l1}
\min_u~\int_{\mathrm\Omega}\frac{|\mbox{det } \nabla^2 u|}{(1+|\nabla u|^2)^2}dx+ \frac{1}{2\lambda}\int_{\mathrm{\Omega}}(u-f)^2dx,
\end{equation}
where $ \nabla^2 u$ is the Hessian of function $u$ and $\mbox{det } \nabla^2 u$ denotes the determinant of Hessian.
It is proven to be with the same geometric properties as mean curvature model \eqref{mcdenoising}. However, due to the highly nonlinearity of the model \eqref{mcdenoising} and \eqref{gauss-curvature-l1}, the minimizations of curvature regularized models are quite challenging. Although ADMM-based algorithms have been developed for the Euler's elastica model \eqref{euler-elastica} and mean curvature model \eqref{mcdenoising}, multiple artificial variables are introduced resulting in more parameters need to be selected manually \cite{tai2011fast,zhu2013augmented}.
The case of Gaussian curvature model \eqref{gauss-curvature-l1} is even more complicated, which was solved by a two-step method based on the vector filed smoothing and gray level interpolation \cite{brito2016image}. By estimating the curvatures explicitly, Zhong, Yin and Duan \cite{Zhong2021image} proposed to minimize certain functions of Gaussian/mean curvature over the image surface, which are solved as a weighted image surface minimization problem with high efficiency.

In this work, we first introduce the \emph{Weingarten map} or \emph{shape operator} of the image surface as the regularization for image reconstruction. We theoretically show that the Weingarten map regularizer can provide good geometric properties including keeping image contrast and preserving edges and corners of objects. The Weingarten map minimization model is solved by the ADMM-based algorithm, where the original nontrivial problem is decomposed into three subproblems. Although two subproblems can be handled with Fast Fourier Transform (FFT) and the closed form solution, the remaining one needs to be solved by gradient descent due to its high nonlinearity. Therefore, to further improve the computational efficiency, we reformulate the Weingarten map into a hybrid nonlinear first and second order regularization. By regarding the nonlinear weights as known, an efficient numerical algorithm is developed based on the proximal ADMM, where all variables can be solved by either FFT or shrinkage operation. 
Numerous experiments on image denoising, deblurring and inpainting are conducted to demonstrate the effectiveness and efficiency of the proposed models by comparing with other well established high order methods.

The rest of the paper is organized as follows. In Sect. \ref{sec2}, we introduce the Weingarten map minimization model and verify its geometric properties in preserving image contrast, edges and corners of objects. Sect. \ref{sec3} is devoted to developing the numerical algorithm for the Weingarten map minimization model. We derive a hybrid nonlinear first and second order regularization from the proposed Weingarten map and discuss its numerical solution in Sect. \ref{sec4}. Sect. \ref{sec5} implements the comprehensive numerical experiments to demonstrate the effectiveness and superiority of the proposed method. We summarize our specific work with a conclusion in Sect. \ref{sec6}.

\subsection*{Notations}
Let ${\rm\Omega}$ be a domain in $\mathbb{R}^n$ and $p$ be a positive real number. We denote $L^p(\rm\Omega)$ as the class of all measurable functions $f: {\rm\Omega}\rightarrow \mathbb{R}$ such that
\begin{equation*}
  L^p({\rm\Omega})=\{f~\big|~\int_{\rm\Omega}|f(x)|^pdx <\infty,~1\leq p\leq\infty\},
\end{equation*}
We also define the norm as $\|f\|_p=(\int_{\rm \Omega}|f(x)|^pdx)^{\frac{1}{p}}$ with $1\leq p<\infty$ and $\|f\|_{\infty}=\sup\limits_{x\in {\rm \Omega}}|f(x)|$. If $p=2$, we denote $V=L^2(\rm\Omega)$. The inner product of two functions $f,g\in V$ is given by $\langle f,g\rangle_V=\int_{\rm\Omega} f(x)g(x)dx$, and the norm $\|f\|_V=\sqrt{\langle f,f\rangle_V}$.
We let $Q_1=V\times V$. Then for $p=(p_1,p_2)\in Q_1$ and $q=(q_1,q_2)\in Q_1$, there are
\[\langle p,q\rangle_{Q_1}=\langle p_1,q_1\rangle_V+\langle p_2,q_2\rangle_V,\]
and
\[\|p\|_{Q_1} = \sqrt{\langle p,p\rangle_{Q_1}}.\]
Suppose $Q_2=V\times V\times V\times V$. Given $v=\begin{pmatrix} v_{11} & v_{12} \\ v_{21} & v_{22} \end{pmatrix}\in Q_2$, $w=\begin{pmatrix} w_{11} & w_{12} \\ w_{21} & w_{22} \end{pmatrix}\in Q_2$, we also define the inner product and norm accordingly
\[\langle v,w\rangle_{Q_2}=\langle v_{11},w_{11}\rangle_V+\langle v_{12},w_{12}\rangle_V+\langle v_{21},w_{21}\rangle_V+\langle v_{22},w_{22}\rangle_V,\]
and
\[\|v\|_{Q_2} = \sqrt{\langle v,v\rangle_{Q_2}}.\]
To conclude this section, we would like to mention the deviations in the following sections may lack rigorous mathematical foundations. To the best of knowledge, the proper functional frameworks to formulate the curvature minimization problems \eqref{euler-elastica}, \eqref{mcdenoising} and \eqref{gauss-curvature-l1}, have not been identified yet, which have to be a subspace of $L^2(\rm\Omega)$. The situation is the same for our Weingarten map minimization problem. Therefore, we will say no more about the proper choice of the functional space for the proposed model.

\section{The Weingarten map minimization model}
\label{sec2}
\subsection{Description of our model}
Consider the level set function $\phi(x,y,z)=z-u(x,y)$, the zero level set of which corresponds to the image surface $\mathcal S =(x,y, u(x,y))\subset{\mathbb R^3}$. The \emph{unit normal} for points on the zero level set $\{(x,y,z): \phi(x,y,z) = 0\}$ is defined as (cf. equation (1.2) in \cite{osher2002level})
\begin{equation}
N_u= \frac{\nabla \phi}{|\nabla \phi|}=\frac{(\nabla u,-1)}{\sqrt{1+|\nabla u|^2}}.
\label{satv}
\end{equation}
The established mean curvature model \eqref{mcdenoising} is derived by minimizing the $L^1$ norm of the divergence of the unit normal.
Indeed, by directly minimizing the $L^1$ norm of the first component of unit normal vector, we have
\begin{align}
\min_u \int_{\mathrm{\Omega}}\bigg|\frac{\nabla u}{\sqrt{1+|\nabla u|^2}}\bigg| dx,
\label{satvm}
\end{align}
which is a nonlinear first order regularization with the denominator measuring the surface area.
The Weingarten map of image surface can be achieved by pursuing the gradient, i.e.,
\begin{equation}
W_u=\nabla \bigg(\frac{\nabla u}{\sqrt{1+|\nabla u|^2}}\bigg)= \nabla \frac{1}{\sqrt{1+|\nabla u|^2}}\otimes \nabla u+ \frac{1}{\sqrt{1+|\nabla u|^2}}\nabla^2 u,
\label{satv2}
\end{equation}where $a\otimes b=ab^T$, $a,b\in\mathbb R^n$, represents the Euclidean outer product.
Moreover, the matrix form of \eqref{satv2} can be given as
\begin{equation*}
  \Large
  W_u=
  \left[
  \begin{array}{cc}
   \frac{(1+u_y^2)u_{xx}-u_xu_yu_{xy}}{(1+u_x^2+u_y^2)^{3/2}} & \frac{(1+u_x^2)u_{xy}-u_xu_yu_{xx}}{(1+u_x^2+u_y^2)^{3/2}} \\
   \frac{(1+u_y^2)u_{xy}-u_xu_yu_{yy}}{(1+u_x^2+u_y^2)^{3/2}}  & \frac{(1+u_x^2)u_{yy}-u_xu_yu_{xy}}{(1+u_x^2+u_y^2)^{3/2}}\\
  \end{array}
\right],
\end{equation*}
which can be formally defined for each point $p\in\mathcal S$ as a linear self-conjugate map
\[W_{p}:= T_{p}\mathcal S \rightarrow T_{p}\mathcal S \]
with $T_p\mathcal S$ denoting the tangent space of $p$. Particularly, the Weingarten map has very good geometric properties, which can be also interpolated as the combination of the \emph{first fundamental form} {\rm{\uppercase\expandafter{\romannumeral1}}} and the \emph{second fundamental form} {\rm{\uppercase\expandafter{\romannumeral2}}} of the image surface, i.e.,
$W_p={\mathrm{\uppercase\expandafter{\romannumeral1}}}^{-1}{\mathrm{\uppercase\expandafter{\romannumeral2}}}.$
According to the differential geometry theory, the eigenvalues of $W_p$ are the two principal curvatures $\kappa_1$, $\kappa_2$ and it follows that
\begin{den}\label{HGWu}
Let $\mathcal S\subset\mathbb{R}^3$ be an oriented surface and $W_p$ be its Weingarten map at a point $p\in \mathcal S$. Then the mean curvature and Gaussian curvature of point $p$ can be defined by
\begin{align*}
  H_p  :=\frac{1}{2}(\kappa_1+\kappa_2)=\frac{1}{2}\mathrm{trace}(W_p) ~~\mbox{and}\quad K_p  :=\kappa_1\kappa_2=\mathrm{det}(W_p).
\end{align*}
\end{den}
Inspired by the success of the mean curvature and Gaussian curvature for image denoising, we propose to minimize the $L^1$ norm of the Weingarten map, that is to consider the following energy functional
\begin{equation}
E(u)=\int_{\mathrm{\Omega}}\bigg|\nabla \frac{1}{\sqrt{1+|\nabla u|^2}}\otimes\nabla u+ \frac{1}{\sqrt{1+|\nabla u|^2}}\nabla^2 u\bigg|_Fdx+\frac{1}{2\lambda}\int_{\mathrm{\Omega}}(u-f)^2dx.
\label{wHM}
\end{equation}

\begin{figure*}[t]
      \begin{center}
			\includegraphics[width=1.00\linewidth]{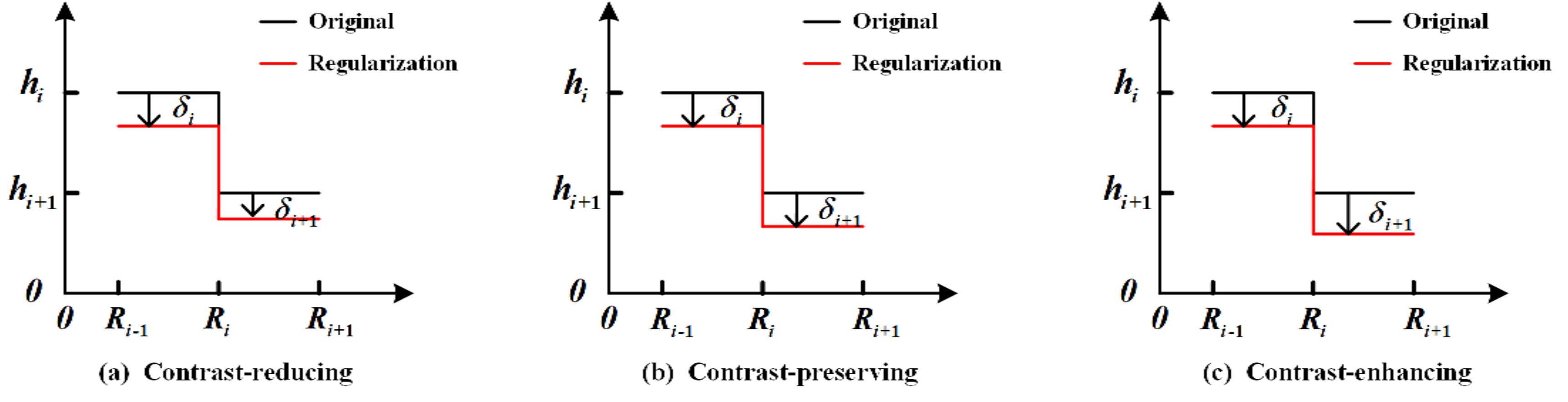}
	  \end{center}
	\caption{Illustration of contrast preservation. (a) Contrast-reducing: the size of discontinuity at $R_i$ after the regularization is reduced, i.e., $\delta_i>\delta_{i+1}$; (b) Contrast-preserving: the size of discontinuity at $R_i$ after the regularization is invariant, i.e., $\delta_i=\delta_{i+1}$; (c) Contrast-enhancing: the size of discontinuity at $R_i$ after the regularization is increased, i.e., $\delta_i<\delta_{i+1}$.}
	\label{contrast}
\end{figure*}

\subsection{Geometric properties of our model}
It is well-known that edges and contrasts are important features for signals and images. Thus, an ideal model for image reconstruction should be able to preserve not only neat edges, but also the contrasts of the edges, i.e., the size of the discontinuities. As discussed in \cite{strong2003edge,wu2018general}, TV regularization suffers from a contrast reduction (see Fig. \ref{contrast}(a)), and only nonconvex and nonsmooth regularization can preserve the image contrast (see Fig. \ref{contrast}(b)). In this subsection, we develop a preliminarily analytical study on the Weingarten map minimization model \eqref{wHM} to verify its contrast-preserving and edge-preserving properties.

Let $f$ be a piecewise constant function defined on a rectangle ${\mathrm{\Omega}}=(-2R,2R)\times(-2R,2R)$ composed of a series of open disks $B(0, R_i)\in\mathbb R^2$ centered at origin with radius $R_i$, for $i=1,\ldots,n$. To identify the subdomains ${\rm\Omega}_i$, $i=1,\ldots,n$, as displayed in Fig. \ref{piecewise-constant-function}(a), we define a piecewise constant level set function
\[ \phi(x)= i,~\mbox{for}~x\in {\rm\Omega}_i, ~~ i = 1,\ldots,n, \]
where ${\rm\Omega}_1 = B(0,R_1)$, and ${\rm\Omega}_{i} = B(0,R_{i})\backslash B(0, R_{i-1})$ for $i=2, \ldots, n$. Associated with such a piecewise constant level set function, the characteristic functions of the subdomains are given as
\[\chi_i = \frac 1{\alpha_i} \prod_{j=1,j\neq i}^n (\phi-j),\quad \alpha_i = \prod_{k=1, k\neq i}^n (i-k),\]
for which we have $\chi_i(x)= 1$ for $x\in{\rm\Omega}_i$ and $\chi_i(x) =0$ elsewhere. We further define $f= \sum_{i=1}^n h_i\chi_i(x,y)$ with $h_i>0$ for $i=1,\ldots,n$. Since $f$ is radial symmetric, it can be obtained by rotating the function of one variable $\hat{f}(x)=\sum_{i=1}^n h_i\chi_{i[0,2R]}(x)$ around the vertical axis. As shown in Fig. \ref{piecewise-constant-function}(b), we approximate $\hat f$ by a sequence of smooth functions $\{u_n\}$. Then we can calculate the integral $\int_{\rm\Omega}|W_{u_n}|dxdy$ and define $\int_{\mathrm{\Omega}} |W_f|_Fdxdy$ to be $\lim_{n\rightarrow +\infty} \int_{\mathrm{\Omega}} |W_{u_n}|_Fdxdy$.

\begin{lem}\label{lemmaHf}
Assume $f= \sum_{i=1}^n h_i\chi_i(x,y)$ be a piecewise constant image defined on a rectangle ${\rm\Omega}=(-2R,2R)\times(-2R,2R)$, where $\chi_i$ is the characteristic function of the subdomain ${\rm\Omega}_i$ and $h_i>0$ for $i=1,\ldots, n$. Note that the subdomains $\{{\rm\Omega}_i\}_{i=1}^n$ are defined by the open disks $B(0, R_i)$, $i=1,\ldots, n$, centered as the origin such that ${\rm\Omega}_1 = B(0,R_1)$ and ${\rm\Omega}_{i} = B(0,R_{i})\backslash B(0, R_{i-1})$ for $i=2, \ldots, n$. Then we obtain
\begin{equation}\label{Hfintegraln}
  \int_{\mathrm{\Omega}} |W_f|_Fdxdy = \sum_{i=1}^{n}4\pi R_{i}.
\end{equation}
\end{lem}

\begin{proof}\label{ProofHf}
Referring to the Lemma 2.1 in \cite{zhu2012image}, we define a sequence of smooth functions $\{u_n\}$ of one variable and rotate their graphs around the vertical axis to generate smooth radial symmetric surfaces, which are used to approximate the surface of $f$. Specially, we consider rotating the curves of a sequence of smooth functions $\{u_n\}$ in the set $\bf S$ defined as
\begin{align*}
  {\bf S} =  \Big\{ u\in {\bf C}^2[0,2R]:&~u''(x)\leq 0,~\mbox{for}~ x\in(R_{i-1},R_i),~ u''(x)\geq 0,~\mbox{for}~x\in(R_i,R_{i+1});\\
  ~&\exists~\varepsilon>0,~R_{i-1}<R_i-\varepsilon<R_i<R_i+\varepsilon<R_{i+1}, R_{0} =0,~ R_{n+1} = 2R, ~\mbox{such that}~\\ &u(x)=\hat{f}(x)~\mbox{if}~x\in(R_{i-1},R_i-\varepsilon],~u(x)=\hat{f}(x)~\mbox{if}~x\in[R_i+\varepsilon,R_{i+1}),~\forall ~1\leq i\leq n;\\
  &u(0)=\hat{f}(0),~ u(2R)=0;~u'(R_i) <-\frac{2h_i}{R_i}\Big\}.
\end{align*}
If $u\in\bf S$, rotating $u$ yields an image surface $z=u(r)$ with $r=\sqrt{x^2+y^2}$. From ${\bf S}$, we can select a sequence of smooth functions to approach the function $\hat{f}$, then obtain a sequence of smooth radial symmetric functions to approximate the target function $f$.

For the radial symmetric surface $z=u(r)=u(\sqrt{x^2+y^2})$, we have
\begin{equation*}
 u_x=u'\frac{x}{r},~u_y=u'\frac{y}{r},~u_{xx}=u''\frac{x^2}{r^2}+u'\frac{y^2}{r^3},
 ~u_{yy}=u''\frac{y^2}{r^2}+u'\frac{x^2}{r^3},~u_{xy}=u''\frac{xy}{r^2}-u'\frac{xy}{r^3}.
\end{equation*}
Therefore, the Weingarten map of a surface $z=u(r)$ takes the following form
\begin{align*}
  W_u  & = \nabla \frac{1}{\sqrt{1+|\nabla u|^2}}\otimes\nabla u+\frac{1}{\sqrt{1+|\nabla u|^2}}{\nabla}^2 u  \\
  & = \Large \left[
  \begin{array}{cc}
   \frac{u''\frac{x^2}{r^2}+u'\frac{y^2}{r^3}(1+(u')^2)}{(1+(u')^2)^{3/2}} &
   \frac{u''\frac{xy}{r^2}-u'\frac{xy}{r^3}(1+(u')^2)}{(1+(u')^2)^{3/2}} \\
   \frac{u''\frac{xy}{r^2}-u'\frac{xy}{r^3}(1+(u')^2)}{(1+(u')^2)^{3/2}}  &
   \frac{u''\frac{y^2}{r^2}+u'\frac{x^2}{r^3}(1+(u')^2)}{(1+(u')^2)^{3/2}}\\
  \end{array}
    \right].
\end{align*}
\begin{figure*}[t]
      \centering
      \subfigure[]{
            \includegraphics[width=0.3\linewidth]{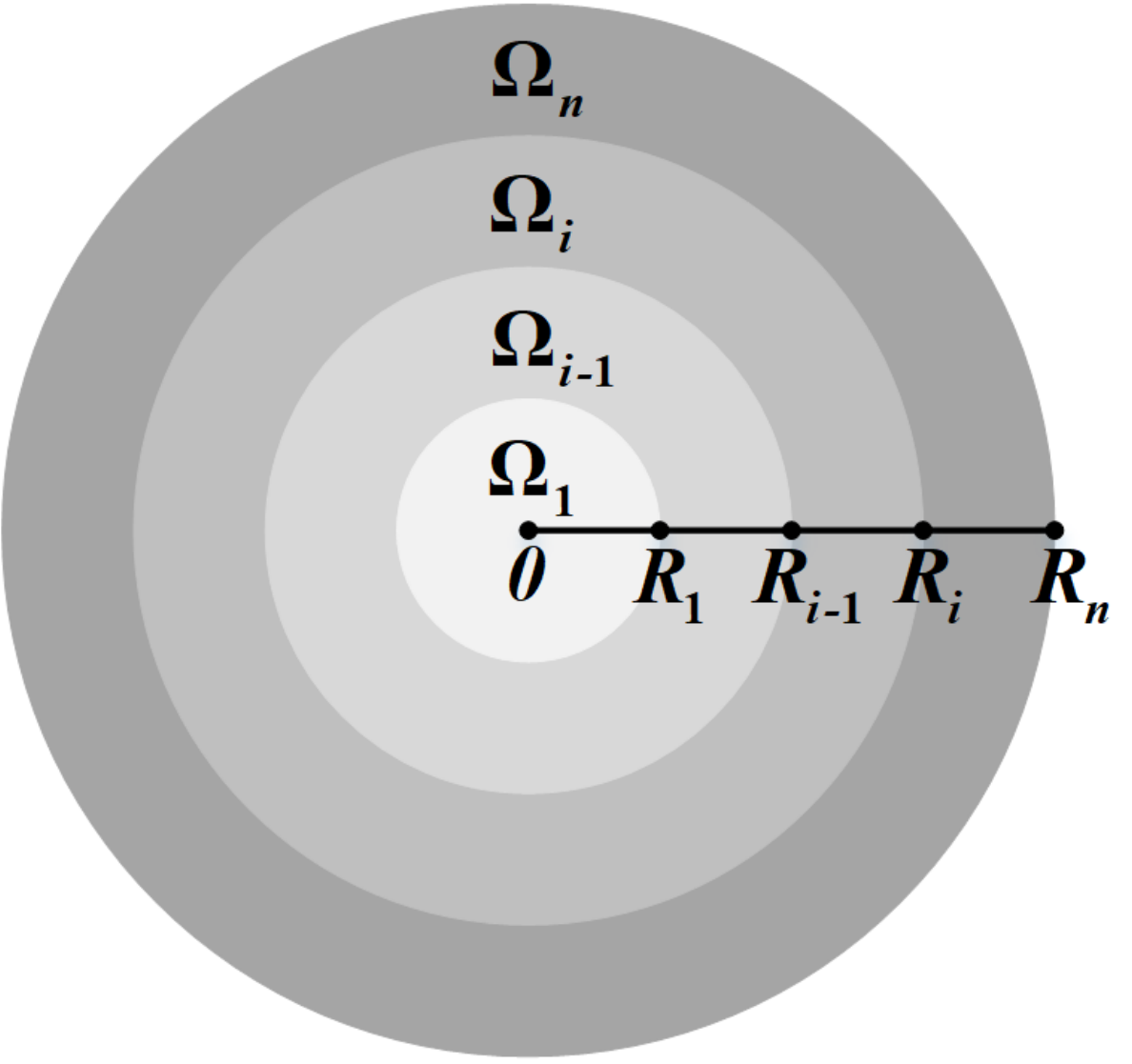}}\hspace{15ex}
      \subfigure[]{
            \includegraphics[width=0.38\linewidth]{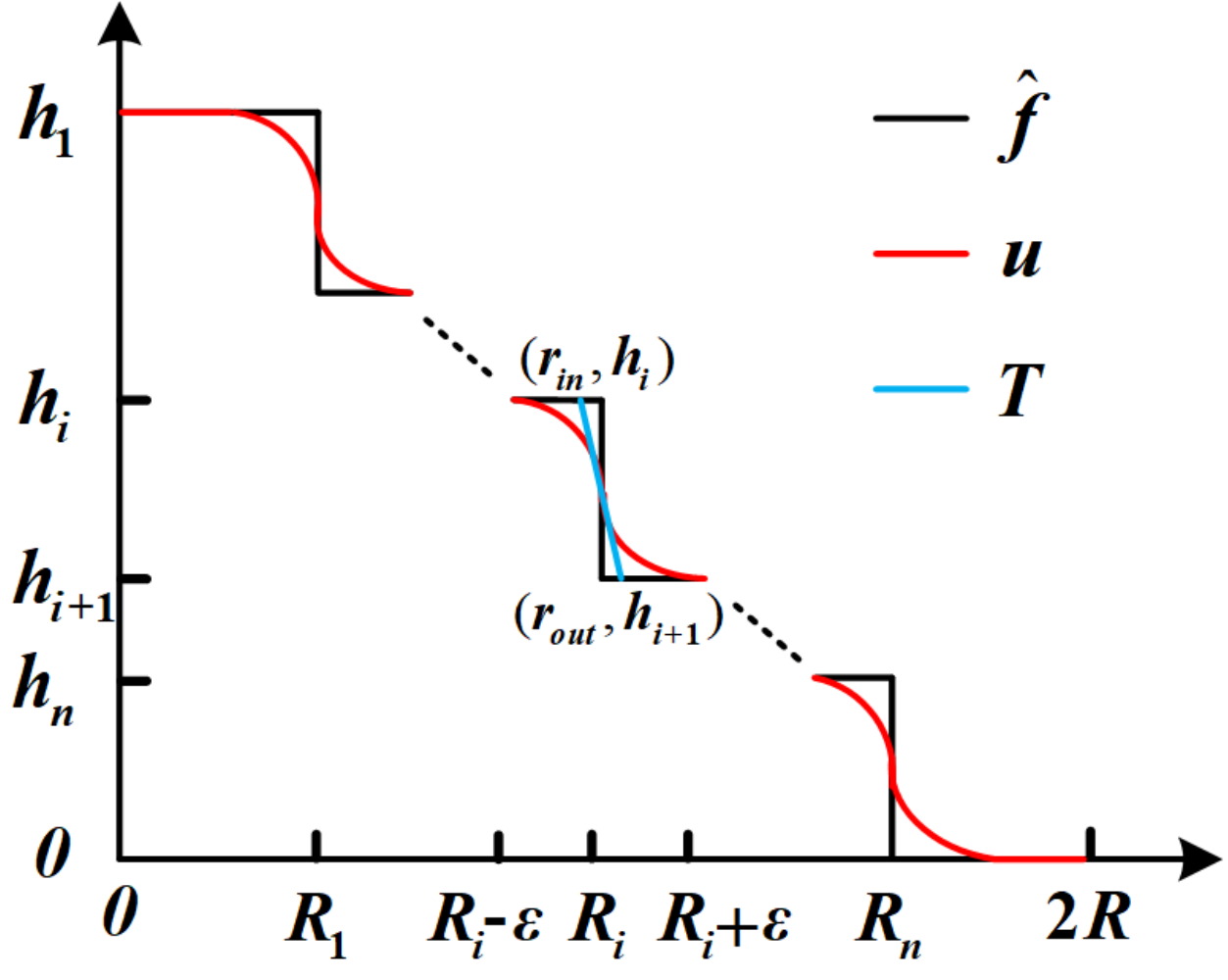}}
	\caption{(a): the piecewise constant function $f$ defined on $\rm\Omega=(-2R,2R)\times(-2R,2R)$; and (b): the generatrix function $u$ that creates the approximation function by rotating around the vertical axis, where $T$ denotes the tangent line of $u$ at point $(R_i,u(R_i))$ in the subdomain $(R_i-{\varepsilon},R_i+{\varepsilon})$.}
	\label{piecewise-constant-function}
\end{figure*}
Furthermore, the Weingarten map regularization can be written as follows
\begin{equation}\label{Hgnorm}
  |W_u|_F  = \sqrt{\Big(\frac{u''}{(\sqrt{1+(u')^2})^3}\Big)^2+\Big(\frac{u'}{r\sqrt{1+(u')^2}}\Big)^2}.
\end{equation}
Correspondingly, we obtain the following result
\begin{align}\label{Hgintegral}
  \int_{\mathrm{\Omega}} |W_{u}|_F dxdy & = \int_0^{2\pi}d\theta \int_0^{2R}r|W_u|_Fdr \nonumber \\
  & = 2\pi\int_0^{2R} r\sqrt{\Big(\frac{u''}{(\sqrt{1+(u')^2})^3}\Big)^2+\Big(\frac{u'}{r\sqrt{1+(u')^2}}\Big)^2}dr.
\end{align}
If $u\in{\bf S}$, one can see that $u''\leq 0$ and $u'\leq 0$ for $r\in(R_i-\varepsilon,R_{i})$, we obtain
\begin{equation*}
 \Big|\frac{u''}{(\sqrt{1+(u')^2})^3}-\frac{u'}{r\sqrt{1+(u')^2}}\Big| \leq |W_u|_F \leq  \Big|\frac{u''}{(\sqrt{1+(u')^2})^3} + \frac{u'}{r\sqrt{1+(u')^2}}\Big|.
\end{equation*}
Note that $\frac{u''}{(\sqrt{1+(u')^2})^3}=\Big[\frac{u'}{\sqrt{1+(u')^2}}\Big]'$ and $\frac{u''}{(\sqrt{1+(u')^2})^3} + \frac{u'}{r\sqrt{1+(u')^2}}=\frac{1}{r}\Big[r\frac{u'}{\sqrt{1+(u')^2}}\Big]'$, one gets
\begin{align*}
 &\int_{R_i-{\varepsilon}}^{R_{i}}r\Big|\frac{u''}{(\sqrt{1+(u')^2})^3}-\frac{u'}{r\sqrt{1+(u')^2}}\Big|dr \geq \int_{R_i-{\varepsilon}}^{R_{i}}r\Big|\frac{u''}{(\sqrt{1+(u')^2})^3}\Big|dr-\int_{R_i-{\varepsilon}}^{R_{i}}\Big|\frac{u'}{\sqrt{1+(u')^2}}\Big|dr   \\
  &=-\int_{R_i-{\varepsilon}}^{R_{i}}r\Big[\frac{u'}{\sqrt{1+(u')^2}}\Big]'dr+\int_{R_i-{\varepsilon}}^{R_{i}}\frac{u'}{\sqrt{1+(u')^2}}dr \\
  &= -\Big(R_{i}\frac{u'(R_{i})}{\sqrt{1+(u'(R_{i}))^2}}- (R_i-{\varepsilon})\frac{u'(R_i-{\varepsilon})}{\sqrt{1+(u'(R_i-{\varepsilon}))^2}}\Big) +2\int_{R_i-{\varepsilon}}^{R_{i}}\frac{u'}{\sqrt{1+(u')^2}}dr,
\end{align*}
and
\begin{align*}
 \int_{R_i-{\varepsilon}}^{R_{i}}r\Big|\frac{u''}{(\sqrt{1+(u')^2})^3} + \frac{u'}{r\sqrt{1+(u')^2}}\Big|dr &= \int_{R_i-{\varepsilon}}^{R_{i}}\Big|\Big[r\frac{u'}{\sqrt{1+(u')^2}}\Big]'\Big|dr = - \int_{R_i-{\varepsilon}}^{R_{i}}\Big[r\frac{u'}{\sqrt{1+(u')^2}}\Big]'dr \\
 &= -\Big(R_{i}\frac{u'(R_{i})}{\sqrt{1+(u'(R_{i}))^2}}- (R_i-{\varepsilon})\frac{u'(R_i-{\varepsilon})}{\sqrt{1+(u'(R_i-{\varepsilon}))^2}}\Big).
\end{align*}
Thus we have
\begin{align}\label{HgRi}
-\Big(R_{i}\frac{u'(R_{i})}{\sqrt{1+(u'(R_{i}))^2}}& - (R_i-{\varepsilon})\frac{u'(R_i-{\varepsilon})}{\sqrt{1+(u'(R_i-{\varepsilon}))^2}}\Big) +2\int_{R_i-{\varepsilon}}^{R_{i}}\frac{u'}{\sqrt{1+(u')^2}}dr \\
&\leq \int_{R_i-{\varepsilon}}^{R_{i}}r|W_u|_Fdr \leq -\Big(R_{i}\frac{u'(R_{i})}{\sqrt{1+(u'(R_{i}))^2}}- (R_i-{\varepsilon})\frac{u'(R_i-{\varepsilon})}{\sqrt{1+(u'(R_i-{\varepsilon}))^2}}\Big).\nonumber
\end{align}
When $r\in(R_{i},R_i+{\varepsilon})$, $u''\geq 0$ and $u'\leq 0$, we get
\begin{equation*}
\Big|\frac{u''}{(\sqrt{1+(u')^2})^3}+\frac{u'}{r\sqrt{1+(u')^2}}\Big| \leq |W_u|_F \leq \Big|\frac{u''}{(\sqrt{1+(u')^2})^3} - \frac{u'}{r\sqrt{1+(u')^2}}\Big|,
\end{equation*}
where
\begin{align*}
 \int_{R_{i}}^{R_i+\varepsilon}r\Big|\frac{u''}{(\sqrt{1+(u')^2})^3}+\frac{u'}{r\sqrt{1+(u')^2}}\Big|dr &\geq\int_{R_{i}}^{R_i+\varepsilon}r\Big[\frac{u'}{\sqrt{1+(u')^2}}\Big]'dr +\int_{R_{i}}^{R_i+\varepsilon}\frac{u'}{\sqrt{1+(u')^2}}dr \\
 &= (R_i+{\varepsilon})\frac{u'(R_i+{\varepsilon})}{\sqrt{1+(u'(R_i+{\varepsilon}))^2}} -R_{i}\frac{u'(R_{i})}{\sqrt{1+(u'(R_{i}))^2}},
\end{align*}
and
\begin{align*}
 &\int_{R_{i}}^{R_i+\varepsilon}r\Big|\frac{u''}{(\sqrt{1+(u')^2})^3} - \frac{u'}{r\sqrt{1+(u')^2}}\Big|dr = \int_{R_{i}}^{R_i+\varepsilon}r\Big[\frac{u'}{\sqrt{1+(u')^2}}\Big]'dr - \int_{R_{i}}^{R_i+\varepsilon}\frac{u'}{\sqrt{1+(u')^2}}dr \\
 &= (R_i+{\varepsilon})\frac{u'(R_{i+\varepsilon})}{\sqrt{1+(u'(R_{i+\varepsilon}))^2}} -R_{i}\frac{u'(R_{i})}{\sqrt{1+(u'(R_{i}))^2}} - 2\int_{R_{i}}^{R_i+\varepsilon}\frac{u'}{\sqrt{1+(u')^2}}dr.
\end{align*}
It follows that
\begin{align}\label{HgRi1}
(R_i+{\varepsilon})\frac{u'(R_i+{\varepsilon})}{\sqrt{1+(u'(R_i+{\varepsilon}))^2}} &-R_{i}\frac{u'(R_{i})}{\sqrt{1+(u'(R_{i}))^2}} \leq \int_{R_{i}}^{R_i+\varepsilon}r|W_u|_Fdr \\
&\leq (R_i+{\varepsilon})\frac{u'(R_i+{\varepsilon})}{\sqrt{1+(u'(R_i+{\varepsilon}))^2}} -R_{i}\frac{u'(R_{i})}{\sqrt{1+(u'(R_{i}))^2}} - 2\int_{R_{i}}^{R_i+\varepsilon}\frac{u'}{\sqrt{1+(u')^2}}dr. \nonumber
\end{align}
When $r\in [0,R_1-\varepsilon]$, $[R_n+\varepsilon,2R]$ and $[R_i+\varepsilon,R_{i+1}-\varepsilon]$, $i=1,\ldots,n-1$, there is $u'=u''=0$.
Thus, based on \eqref{Hgintegral}, by adding the formulas \eqref{HgRi} and \eqref{HgRi1}, we obtain the following inequalities
\begin{align}\label{Hginequality}
&2\pi\sum_{i=1}^{n}\Big(-2R_{i}\frac{u'(R_{i})}{\sqrt{1+(u'(R_{i}))^2}} +(R_{i}-\varepsilon)\frac{u'(R_{i}-\varepsilon)}{\sqrt{1+(u'(R_{i}-\varepsilon))^2}} +(R_i+{\varepsilon})\frac{u'(R_i+{\varepsilon})}{\sqrt{1+(u'(R_i+{\varepsilon}))^2}} \nonumber \\ &+2\int_{R_i-{\varepsilon}}^{R_{i}}\frac{u'}{\sqrt{1+(u')^2}}dr\Big)
\leq \int_{\mathrm{\Omega}} |W_{u}|_Fdxdy \leq 2\pi\sum_{i=1}^{n}\Big(-2R_{i}\frac{u'(R_{i})}{\sqrt{1+(u'(R_{i}))^2}} \nonumber  \\ &+(R_{i}-\varepsilon)\frac{u'(R_{i}-\varepsilon)}{\sqrt{1+(u'(R_{i}-\varepsilon))^2}} +(R_i+{\varepsilon})\frac{u'(R_i+{\varepsilon})}{\sqrt{1+(u'(R_i+{\varepsilon}))^2}}  -2\int_{R_{i}}^{R_i+{\varepsilon}}\frac{u'}{\sqrt{1+(u')^2}}dr\Big).
\end{align}
Considering $\{u_n\}\in{\bf S}$ is any sequence of functions that pointwise converge to $\hat{f}$, it is easy to obtain that $u'_n(R_{i})\rightarrow -\infty$ and $u'_n(r)\rightarrow 0$ with $r\neq R_{i}$ when $n\rightarrow +\infty$. In addition, through the dominated convergence theorem, we have
\begin{align}\label{gnlimit}
\lim_{n\rightarrow +\infty} (R_i-{\varepsilon})\frac{u_n'(R_i-{\varepsilon})}{\sqrt{1+(u_n'(R_i-{\varepsilon}))^2}} &= \lim_{n\rightarrow +\infty} (R_i+{\varepsilon})\frac{u_n'(R_i+{\varepsilon})}{\sqrt{1+(u_n'(R_i+{\varepsilon}))^2}} \\
&= \lim_{n\rightarrow +\infty} \int_{R_{i}-\varepsilon}^{R_i}\frac{u'_n}{\sqrt{1+(u'_n)^2}}dr = \lim_{n\rightarrow +\infty} \int_{R_i}^{R_{i}+\varepsilon}\frac{u'_n}{\sqrt{1+(u'_n)^2}}dr=0. \nonumber
\end{align}
Moreover, according to the inequalities \eqref{Hginequality}, there is
\begin{equation*}
 \lim_{n\rightarrow +\infty} \int_{\mathrm{\Omega}} |W_{u_n}|_Fdxdy = \int_{\mathrm{\Omega}} |W_{f}|_Fdxdy = \sum_{i=1}^{n} 4\pi R_{i}. ~\quad\blacksquare
\end{equation*}
\end{proof}

\begin{remark}
The integral of Weingarten map is similar to mean curvature \cite{zhu2012image}, both of which do not rely on image intensities. It describes an important characteristic of the Weingarten map regularizer, which motivates the following theorem showing the model \eqref{wHM} with properties of edge and contrast preservation.
\end{remark}

\begin{remark}
For a broad class of $f=h\chi_A(x,y)$ with $A\subset\mathrm{\Omega}$ being an arbitrary open set with $C^2$ boundary, one can easily obtain
\begin{equation*}
\int_{\mathrm{\Omega}} |\nabla f|dxdy  = \sup_{\substack{p\in{C_c^1}(\mathrm{\Omega},\mathbb{R}^{n})\\ \|p\|_\infty\leq 1}} \int_{\rm\Omega} f\div p dxdy = \sup_{\substack{p\in{C_c^1}(\mathrm{\Omega},\mathbb{R}^{n})\\ \|p\|_\infty\leq 1}} \int_{\partial A} f p\cdot \nu d\mathcal{H}^1 = hPer(A,\mathrm{\Omega}),
\end{equation*}
where $\nu$ is the exterior normal to $\partial A$, $\mathcal{H}^1$ is the one-dimensional Hausdorff measure, and $Per(A,\mathrm{\Omega})$ denotes the perimeter of $A$ inside $\mathrm{\Omega}$. As can be seen, the integral of total variation is related to the intensity $h$ \cite{chambolle2010introduction}. Likewise, we show the integral of Weingarten map is independent of image intensity. More details can be found in Appendix A. Besides, we also extend the discussion to the multiphase piecewise constant function defined on $\mathrm{\Omega}$ with $C^2$ boundary; see Appendix B for details.
\end{remark}

Based on Lemma \ref{lemmaHf}, we can further prove that $f=\sum_{i=1}^n h_i\chi_i(x,y)$ is a minimizer of the proposed model \eqref{wHM} as long as $\lambda$ being small enough, which means our model can preserve image contrast and edges for image restoration.

\begin{thm}\label{Theoremcontrast}
Let $f= \sum_{i=1}^n h_i\chi_i(x,y)$ be a piecewise constant image defined on a rectangle ${\rm\Omega}=(-2R,2R)\times(-2R,2R)$. Then there exists a constant $C$ such that if $\lambda<C$, $f$ attains the infimum of the proposed model \eqref{wHM} inside the function set $\bf S$, that is
$E(f)={\inf}_{u\in{\bf S}}E(u)$.
\end{thm}

\begin{proof}\label{Proofcontrast}
As shown in Fig. \ref{piecewise-constant-function}(b), we can draw the tangent line $T(r)=u(R_i)+u'(R_i)(r-R_i)$ to any $u\in{\bf S}$ at the point $(R_i,u(R_i))$ in the subdomain $(R_i-\varepsilon,R_i+\varepsilon)$. Suppose the tangent line intersects $\hat{f}$ at point $(r_{in},h_i)$ and $(r_{out},h_{i+1})$, then we get $r_{in}=R_i+(h_i-u(R_i))/u'(R_i)$ and $r_{out}=R_i+(h_{i+1}-u(R_i))/u'(R_i)$ with $r_{in}\in(R_i-\varepsilon,R_i)$ and $r_{out}\in(R_{i},R_i+\varepsilon)$, respectively. Moreover, there is $\hat{f}(r)-u(r)\geq \hat{f}(r)-T(r)>0$ if $r\in(r_{in},R_i)$, and $u(r)-\hat{f}(r)\geq T(r)-\hat{f}(r)>0$ if $r\in(R_i,r_{out})$. Thus we obtain
\begin{align}\label{upbound}
  \int_{\mathrm{\Omega}}(u-f)^2dxdy & = \sum_{i=1}^{n}\Big(\int_{0}^{2\pi}d\theta \int_{R_i-\varepsilon}^{R_i+\varepsilon}(u(r)-\hat{f}(r))^2rdr\Big) \nonumber \\
  & \geq \sum_{i=1}^{n}2\pi \Big[\int_{r_{in}}^{R_i}(\hat{f}(r)-u(r))^2rdr+\int_{R_i}^{r_{out}}(\hat{f}(r)-u(r))^2rdr\Big] \nonumber \\
   & \geq \sum_{i=1}^{n}2\pi\Big[\int_{r_{in}}^{R_i}(h_i-T(r))^2rdr+\int_{R_i}^{r_{out}}(h_{i+1}-T(r))^2rdr\Big].
\end{align}
It is easy to calculate the following integrals
\begin{align*}
  &\int_{r_{in}}^{R_i}(h_{i}-T(r))^2rdr = \int_{r_{in}}^{R_i}(h_{i}-u(R_i)-u'(R_i)(r-R_i))^2rdr \\
  & = -\frac{1}{3u'(R_i)}\Big(r[h_{i}-u(R_i)-u'(R_i)(r-R_i)]^3|_{r=r_{in}}^{r=R_i}-\int_{r_{in}}^{R_i}(h_{i}-u(R_i) -u'(R_i)(r-R_i))^3dr\Big) \\
  & = -\frac{1}{3u'(R_i)}\Big(R_i(h_i-u(R_i))^3 + \frac{1}{4u'(R_i)}[h_{i}-u(R_i)-u'(R_i)(r-R_i)]^4|_{r=r_{in}}^{r=R_i}\Big)\\
  & = -\frac{1}{3u'(R_i)}R_i(h_i-u(R_i))^3-\frac{1}{12(u'(R_i))^2}(h_i-u(R_i))^4,
\end{align*}
and
\begin{equation*}
 \int_{R_i}^{r_{out}}(h_{i+1}-T(r))^2rdr = \frac{1}{3u'(R_i)}R_i(h_{i+1}-u(R_i))^3+\frac{1}{12(u'(R_i))^2}(h_{i+1}-u(R_i))^4.
\end{equation*}
Then it follows that
\begin{align*}
   &\int_{r_{in}}^{R_i}(h_{i}-T(r))^2rdr + \int_{R_i}^{r_{out}}(h_{i+1}-T(r))^2rdr \\
   & = \frac{1}{3u'(R_i)}R_i[(h_{i+1}-u(R_i))^3-(h_i-u(R_i))^3] +\frac{1}{12(u'(R_i))^2}[(h_{i+1}-u(R_i))^4-(h_i-u(R_i))^4]\\
   & \geq \frac{1}{3u'(R_i)}R_i \frac{(h_{i+1}-h_{i})^3}{4}-\frac{1}{12(u'(R_i))^2}(h_{i}-h_{i+1})^4 = -\frac{(h_{i}-h_{i+1})^3}{12u'(R_i)}\Big(R_i+\frac{h_{i}-h_{i+1}}{u'(R_i)}\Big).
\end{align*}
Since $u\in{\bf S}$, $u'(R_i)<-\frac{2h_i}{R_i}$, then $R_i+\frac{h_{i}-h_{i+1}}{u'(R_i)}>\frac{(h_{i}+h_{i+1})R_i}{2h_i}$. Thus, we obtain
\begin{equation}\label{databound}
 \int_{\mathrm{\Omega}}(u-f)^2dxdy \geq \sum_{i=1}^{n} -\frac{\pi(h_{i}-h_{i+1})^3(h_{i}+h_{i+1})}{12h_iu'(R_i)}R_i.
\end{equation}
Based on the formulas \eqref{Hginequality} and \eqref{databound}, we have
\begin{align}\label{Eg}
  E(u) &  =\int_{\mathrm{\Omega}}|W_u|_Fdxdy+\frac{1}{2\lambda}\int_{\mathrm{\Omega}}(u-f)^2dxdy \nonumber \\
  & > \sum_{i=1}^{n}\Big(-4\pi R_i\frac{u'(R_i)}{\sqrt{1+(u'(R_i))^2}} - \frac{\pi(h_{i}-h_{i+1})^3(h_{i}+h_{i+1})}{24h_i\lambda u'(R_i)}R_i\Big) \nonumber \\
  & = \sum_{i=1}^{n}\Big(4\pi R_i \frac{(-u'(R_i))}{\sqrt{1+(-u'(R_i))^2}} + \frac{\pi(h_{i}-h_{i+1})^3(h_{i}+h_{i+1})}{24h_i\lambda (-u'(R_i))}R_i\Big).
\end{align}
For each term in \eqref{Eg}, by defining $\ell=-u'(R_i)$, $c_1=4\pi R_i$ and $c_2=\frac{\pi(h_{i}-h_{i+1})^3(h_{i} + h_{i+1})R_i}{24h_i}$, we consider the function $\eta(\ell)=\frac{c_1\ell}{\sqrt{1+\ell^2}}+\frac{c_2}{\lambda \ell}$ defined on $[\frac{2h_i}{R_i},+\infty)$. Then there is
\begin{align*}
\eta'(\ell) = \frac{c_1}{({1+\ell^2})^{3/2}}-\frac{c_2}{\lambda \ell^2} \leq \frac{c_1}{\ell^3}(1-\frac{c_2}{\lambda c_1}\ell).
\end{align*}
If $\lambda<c_i=\frac{c_2}{c_1}\frac{2h_i}{R_i}$, $\eta'(\ell)<0$ for any $\ell\in [\frac{2h_i}{R_i},+\infty)$. One can see that $\lim_{\ell\rightarrow +\infty}\eta(\ell)=4\pi R_i$, which means $\eta(\ell)$ will strictly decrease to $4\pi R_i$ on $[\frac{2h_i}{R_i},+\infty)$.

Suppose $C=\min\{c_i~|~c_i=\frac{(h_{i}-h_{i+1})^3(h_{i}+h_{i+1})}{48R_i},~i=1,\ldots, n\}$. When $\lambda < C$ in the model \eqref{wHM}, $E(u)>\sum_{i=1}^{n}4\pi R_i=E(f)$ for any smooth function $u\in{\bf S}$. Moreover, for any small $\varepsilon>0$, one can easily find a smooth function $u\in{\bf S}$ satisfying $E(u)-\varepsilon <E(f)<E(u)$. Thus, we obtain $E(f)={\inf}_{u\in{\bf S}}E(u)$. This demonstrates that the proposed model can keep the image contrast when $\lambda$ is small enough. $\quad\blacksquare$
\end{proof}

\begin{remark}
This theorem indicates that the proposed model \eqref{wHM} can keep the image contrast once $\lambda$ is small enough. In contrast, according to \cite{strong2003edge}, the Rudin-Osher-Fetami model will lose image contrast no matter how small $\lambda$ is.
\end{remark}

\begin{remark}
The theorem also indicates that our proposed model \eqref{wHM}, similar to the Rudin-Osher-Fetami model, can keep sharp edges, which is another important property for image denoising.
\end{remark}

The image patches can be categorized into homogeneous regions, edges, corners and T-junctions  \cite{Chan2002Euler}. Thus, we turn to discuss whether our model can keep corners of objects. Considering a particular image $f=h{\chi_{\rm\Gamma}}(x,y)$ defined on a rectangle ${\mathrm{\Omega}}=(-R,R)\times(-R,R)$ with ${\rm\Gamma}=(0,R)\times(0,R)$, we calculate the integral $\int_{\mathrm{\Omega}}|W_f|_Fdxdy$ to prove the Weingarten map can preserve corners through the following lemma.

\begin{lem}\label{wfclem}
Let $f=h{\chi_{\rm\Gamma}}(x,y)$ be a sharp image defined on ${\mathrm{\Omega}}=(-R,R)\times(-R,R)$ with ${\rm\Gamma}=(0,R)\times(0,R)$. Then we obtain
\begin{equation}\label{Wf}
 \int_{\mathrm{\Omega}}|W_f|_Fdxdy = 4R.
\end{equation}
\end{lem}

\begin{proof}
Similarly, we introduce a sequence of smooth functions $\{u_n\}$ to approximate $f$. First, we consider a function set $\bf P$ defined as follows
\begin{align*}
  {\bf P}  = \Big\{ \rho \in {\bf C}^2(\mathbb{R}):&\rho(x)=0 ~{\mathrm {if}}~ x<-1, \rho(x)=1 ~{\mathrm {if}}~ x>1; \\
          & \rho'' \geq 0~{\mathrm {in}}~(-1,0), \rho'' \leq 0~{\mathrm {in}}~(0,1); ~{\mathrm {and}}~1 \leq \rho'(0) \leq 2 \Big\}
\end{align*}
and define $\zeta_{\rho,\epsilon}(x,y)$ in terms of $\rho$ through
\begin{eqnarray}\label{zeta}
\zeta_{\rho,\epsilon}(x,y) =
\left \{
\begin{split}
&h\rho(\frac{2y}{\epsilon}),~~~~~~~(x,y)\in[\epsilon,R)\times(-R,R),\\
&h\rho(\frac{2x}{\epsilon}),~~~~~~~(x,y)\in(-R,\epsilon)\times[\epsilon,R),\\
&h\rho(2-\frac{2r}{\epsilon}),~~(x,y)\in(-R,\epsilon)\times(-R,\epsilon),\\
\end{split}
\right.
\end{eqnarray}
with $r=\sqrt{(x-\epsilon)^2+(y-\epsilon)^2}$. Moreover, with the function $\zeta_{\rho,\epsilon}$, we define a function set $\bf Q$ by
\begin{equation*}
 {\bf Q} = \Big\{ \zeta_{\rho,\epsilon}: \rho \in {\bf P}, \epsilon\in(0,\frac{R}{2}) \Big\}.
\end{equation*}
Let $u\in{\bf Q}$, then there exists $\rho\in{\bf P}$ and a small enough $\epsilon$ such that $u=\zeta_{\rho,\epsilon}$. Thus we can construct a convenient sequence of smooth functions $\{u_n\}$ to approximate the surface of $f$. The constructed surface $z=\zeta_{\rho,\epsilon}(x,y)$ will be sufficiently sharp around the edges $\{x=0,y\in[\epsilon,R)\}$, $\{y=0,x\in[\epsilon,R)\}$ and the corner $(0,0)$.

In particular, we can calculate the Weingarten map on image surface $z=\zeta_{\rho,\epsilon}(x,y)$ as follows
\begin{eqnarray}\label{wu}
{W_u} =
\left \{
\begin{split}
&\Big(\frac{\frac{2h}{\epsilon}\rho'(\frac{2y}{\epsilon})}{\sqrt{1+[\frac{2h}{\epsilon}\rho'(\frac{2y}{\epsilon})]^2}}\Big)_y,~~~~(x,y)\in[\epsilon,R)\times(-R,R),\\
&\Big(\frac{\frac{2h}{\epsilon}\rho'(\frac{2x}{\epsilon})}{\sqrt{1+[\frac{2h}{\epsilon}\rho'(\frac{2x}{\epsilon})]^2}}\Big)_x,~~~~(x,y)\in(-R,\epsilon)\times[\epsilon,R),\\
&\sqrt{\Big(\frac{\frac{-2h}{\epsilon}\rho'(2-\frac{2r}{\epsilon})}{\sqrt{1+[\frac{-2h}{\epsilon}\rho'(2-\frac{2r}{\epsilon})]^2}}\Big)_r^2+ \Big(\frac{\frac{-2h}{\epsilon}\rho'(2-\frac{2r}{\epsilon})}{r\sqrt{1+[\frac{-2h}{\epsilon}\rho'(2-\frac{2r}{\epsilon})]^2}}\Big)^2},~(x,y)\in(-R,\epsilon)\times(-R,\epsilon),\\
\end{split}
\right.
\end{eqnarray}
where $A_\varrho=\frac{dA}{d\varrho}$.

Due to $\rho''\geq0$ for $(x,y)\in[\epsilon,R)\times(-R,0)$ and $\rho''\leq0$ for $(x,y)\in[\epsilon,R)\times(0,R)$, it follows that
\begin{equation*}
  \int_{[\epsilon,R)\times(-R,R)}|W_u|_Fdxdy =  \int_\epsilon^R\Big[\int_{-R}^0 W_udy - \int_{0}^R W_udy\Big]dx = \frac{\frac{2h}{\epsilon}\rho'(0)}{\sqrt{1+[\frac{2h}{\epsilon}\rho'(0)]^2}}2(R-\epsilon).
\end{equation*}
And $\rho''\geq0$ for $(x,y)\in(-R,0)\times[\epsilon,R)$ and $\rho''\leq0$ for $(x,y)\in(0,\epsilon)\times[\epsilon,R)$, thus
\begin{equation*}
  \int_{(-R,\epsilon)\times[\epsilon,R)}|W_u|_Fdxdy =  \frac{\frac{2h}{\epsilon}\rho'(0)}{\sqrt{1+[\frac{2h}{\epsilon}\rho'(0)]^2}}2(R-\epsilon).
\end{equation*}
Similar to the inequalities \eqref{Hginequality}, one obtains
\begin{align*}
  \frac{\frac{2h}{\epsilon}\rho'(0)}{\sqrt{1+[\frac{2h}{\epsilon}\rho'(0)]^2}}\pi\epsilon - \pi\int_{0}^\epsilon \frac{\frac{2h}{\epsilon}\rho'(2-\frac{2r}{\epsilon})}{\sqrt{1+[\frac{2h}{\epsilon}\rho'(2-\frac{2r}{\epsilon})]^2}}dr & \leq  \int_{(-R,\epsilon)\times(-R,\epsilon)}|W_u|_Fdxdy \\
   & \leq \frac{\frac{2h}{\epsilon}\rho'(0)}{\sqrt{1+[\frac{2h}{\epsilon}\rho'(0)]^2}}\pi\epsilon + \pi\int_\epsilon^R \frac{\frac{2h}{\epsilon}\rho'(2-\frac{2r}{\epsilon})}{\sqrt{1+[\frac{2h}{\epsilon}\rho'(2-\frac{2r}{\epsilon})]^2}}dr.
\end{align*}
Let $\{u_n=\zeta_{\rho_n,\epsilon_n}\}\subset{\bf Q}$ being any sequence of functions that approximate $f=h{\chi_{\rm\Gamma}}(x,y)$. It is obvious that $\frac{2h}{\epsilon_n}\rho'_n(0)\rightarrow\infty$ as $n\rightarrow\infty$. Then we obtain
\begin{align*}
 \lim_{n\rightarrow\infty}\int_{\mathrm{\Omega}}|W_{u_n}|_Fdxdy &= \lim_{n\rightarrow\infty} \Big\{ \int_{[\epsilon_n,R)\times(-R,R)}|W_{u_n}|_Fdxdy +\int_{(-R,\epsilon_n)\times[\epsilon_n,R)}|W_{u_n}|_Fdxdy \\
  &+\int_{(-R,\epsilon_n)\times(-R,\epsilon_n)}|W_{u_n}|_Fdxdy \Big\} = 4R.
\end{align*}
Therefore, we have $\int_{\mathrm{\Omega}}|W_{f}|_Fdxdy=\lim_{n\rightarrow\infty}\int_{\mathrm{\Omega}}|W_{u_n}|_Fdxdy =4R$, that completes the proof.  $\quad\blacksquare$
\end{proof}

This lemma also illustrates that the integral $\int_{\mathrm{\Omega}}|W_{f}|_Fdxdy$ does not rely on the image intensity $h$. Then, we show our model \eqref{wHM} can preserve the corner of the image $f=h{\chi_{\rm\Gamma}}(x,y)$ followed the same procedure as before.

\begin{thm}\label{wfcthm}
Let $f=h{\chi_{\rm\Gamma}}(x,y)$ be an image defined on a rectangle ${\mathrm{\Omega}}=(-R,R)\times(-R,R)$ with ${\rm\Gamma}=(0,R)\times(0,R)$. Then there exists a constant $C$ such that if $\lambda<C$, $f$ attains the infimum of the proposed model \eqref{wHM}, i.e., $E(f)=\inf_{u\in {\bf Q}}E(u)$.
\end{thm}

\begin{proof}
Based on Lemma \ref{wfclem} and Theorem 2.4 in \cite{zhu2012image}, we have
\begin{align}\label{Eu}
E(u)&=\int_{\mathrm{\Omega}}\bigg|\nabla \frac{1}{\sqrt{1+|\nabla u|^2}}\otimes\nabla u+ \frac{1}{\sqrt{1+|\nabla u|^2}}\nabla^2 u\bigg|_Fdxdy+\frac{1}{2\lambda}\int_{\mathrm{\Omega}}(u-f)^2dxdy  \nonumber \\
& > \frac{(\frac{2h}{\epsilon})\rho'(0)}{\sqrt{1+[(\frac{2h}{\epsilon})\rho'(0)]^2}}[4(R-\epsilon)+\pi\epsilon] +\frac{1}{2\lambda}\Big(\frac{h^3}{6(\frac{2h}{\epsilon})\rho'(0)}(R-\epsilon)+
\frac{\pi h^3}{48(\frac{2h}{\epsilon})\rho'(0)}\epsilon \Big).
\end{align}
Setting $\tau=2h\rho'(0)$ for the above inequality \eqref{Eu}, it follows that
\begin{equation*}
  \eta(\epsilon)=\frac{(\frac{\tau}{\epsilon})}{\sqrt{1+(\frac{\tau}{\epsilon})^2}}[4(R-\epsilon)+\pi\epsilon]+ \frac{h^3}{12\lambda(\frac{\tau}{\epsilon})}(R-\epsilon)+\frac{\pi h^3}{96\lambda(\frac{\tau}{\epsilon})}\epsilon.
\end{equation*}
It is obvious that $\eta(\epsilon)$ goes to $4R=E(f)=\int_{\mathrm{\Omega}}|W_{f}|_Fdxdy$ as $\epsilon\rightarrow 0$. Moreover, by $\epsilon\in(0,\frac{R}{2})$,
\begin{align*}
  \eta'(\epsilon)&=-\frac{\tau\epsilon}{(\epsilon^2+\tau^2)^{3/2}}[4(R-\epsilon)+\pi\epsilon]+\frac{\tau}{(\epsilon^2+\tau^2)^{1/2}}(\pi-4) +\frac{h^3}{12\lambda\tau}(R-2\epsilon)+\frac{\pi h^3}{48\lambda\tau}\epsilon \\
  &\geq - \frac{1}{2\sqrt{\epsilon^2+\tau^2}}[4(R-\epsilon)+\pi\epsilon] +\frac{\tau}{\sqrt{\epsilon^2+\tau^2}}(\pi-4)+\frac{h^3}{12\lambda\tau}(R-2\epsilon+\frac{\pi}{4}\epsilon) \\
  &>-\frac{1}{2\tau}[4(R-\epsilon)+\pi\epsilon]+\pi-4+\frac{h^3}{12\lambda\tau}\frac{\pi R}{8}\\
  &>-\frac{2R}{\tau}+\pi-4+\frac{h^3}{12\lambda\tau}\frac{\pi R}{8},
\end{align*}
thus we choose
\begin{equation}\label{lambda2}
  \lambda<\frac{\pi R h^3}{192R+(4-\pi)192h\rho'(0)}
\end{equation}
to satisfy $\eta'(\epsilon)>0$. Note that $1\leq\rho'(0)\leq2$, so we can set $C=\frac{\pi R h^3}{192R+(4-\pi)384h}$. If $\lambda<C$, for any $\rho\in{\bf P}$, $E(\zeta_{\rho,\epsilon})$ will decrease to $E(f)$ as $\epsilon\rightarrow 0$. Therefore, we obtain $E(u)>E(f)$ for any $u\in {\bf Q}$.

Moreover, for any small $\epsilon>0$, we can find a smooth function $u\in {\bf Q}$ such that $E(f)>E(u)-\epsilon$. That verifies $E(f)=\inf_{u\in{\bf Q}}E(u)$.  $\quad\blacksquare$
\end{proof}

\begin{remark}
This theorem denotes that our model \eqref{wHM} can preserve corners as long as the tuning parameter $\lambda$ is small enough. This is another important feature of our model.
\end{remark}
\begin{remark}
Similar to \cite{zhu2012image}, the discussion in Theorem \ref{Theoremcontrast} and Theorem \ref{wfcthm} hold for a small class of functions with $C^2$ boundaries. In fact, a thorough analysis of the Weingarten map minimization model \eqref{wHM} needs to be considered in an appropriate function space such as $\mathrm{BV}(\mathrm\Omega)$, which is remained as our future work for exploration.
\end{remark}

\section{The ADMM algorithm for Weingarten map minimization}
\label{sec3}

Although the Weingarten map minimization model \eqref{wHM} has very good geometric features, it involves high order derivatives, which result in the difficulties in developing effective and efficient algorithms for solving it numerically. Here, we first rewrite \eqref{wHM} in terms of $\nabla u$ and $\nabla^2 u$ as follows
\begin{equation}\label{coupleformula}
\min_u~\int_{\rm\Omega}\Big| \nabla^2 u\Big(\frac{(1+|\nabla u|^2)\mathcal{I}-\nabla u\otimes \nabla u}{(1+|\nabla u|^2)^\frac{3}{2}}\Big)\Big|_Fdx + \frac{1}{2\lambda}\int_{\rm\Omega}(u-f)^2dx,
\end{equation}
where $\mathcal{I}$ denotes the identity matrix. By introducing two auxiliary variables $v$ and $w$, we can reformulate the above minimization problem into the following equivalent constrained problem
\begin{equation}\label{constrainedproblem}
\begin{split}
& \min_{(u,v,w)\in V\times Q_1\times Q_2}~\int_{\mathrm{\Omega}}\Big|w\Big(\frac{(1+|v|^2)\mathcal{I}-v\otimes v}{(1+|v|^2)^\frac{3}{2}}\Big)\Big|_Fdx+\frac{1}{2\lambda}\int_{\mathrm{\Omega}}(u-f)^2dx \\
& \qquad\quad\mathrm{s.t.}\qquad~~ v=\nabla u,~ w={\nabla}^2 u.
\end{split}
\end{equation}
The corresponding augmented Lagrangian functional can be defined as follows
\begin{align}\label{Lagranfunction}
 \mathcal{L}({u,v,w; \lambda_1,\lambda_2}) =&\int_{\mathrm{\Omega}}\Big|w\Big(\frac{(1+|v|^2)\mathcal{I}-v\otimes v}{(1+|v|^2)^\frac{3}{2}}\Big)\Big|_Fdx +\frac{1}{2\lambda}\int_{\mathrm{\Omega}}(u-f)^2dx -\int_{\mathrm{\Omega}}\lambda_1(v-\nabla u) dx \nonumber\\
 & + \frac{r_1}{2}\int_{\mathrm{\Omega}}\big(v-\nabla u\big)^2dx - \int_{\mathrm{\Omega}}\lambda_2(w-{\nabla}^2 u)dx + \frac{r_2}{2}\int_{\mathrm{\Omega}}\big(w-{\nabla}^2 u\big)^2dx,
\end{align}
where $(\lambda_1$, $\lambda_2)\in Q_1\times Q_2$ are the Lagrange multipliers, and $r_1$, $r_2$ are the penalty parameters. Then we can exploit the ADMM to solve the above saddle-point problem by minimizing the primal variables $u$, $v$ and $w$ from
\begin{eqnarray*}\label{allsubproblems}
\left \{
\begin{split}
&\min_{u\in V}~\frac{1}{2\lambda}\int_{\mathrm{\Omega}}(u-f)^2dx+\frac{r_1}{2}\int_{\mathrm{\Omega}}\big(\nabla u-(v^k-\frac{\lambda_1^k}{r_1})\big)^2dx +\frac{r_2}{2}\int_{\mathrm{\Omega}}\big({\nabla}^2 u-(w^k-\frac{\lambda_2^k}{r_2})\big)^2dx,\\
&\min_{v\in{Q_1}}~\int_{\mathrm{\Omega}}\Big|w\Big(\frac{(1+|v|^2)\mathcal{I}-v\otimes v}{(1+|v|^2)^\frac{3}{2}}\Big)\Big|_Fdx +\frac{r_1}{2}\int_{\mathrm{\Omega}}\big(v-\nabla u^{k+1}-\frac{\lambda_1^k}{r_1}\big)^2dx,\\
&\min_{w\in{Q_2}}~\int_{\mathrm{\Omega}}\Big|w\Big(\frac{(1+|v|^2)\mathcal{I}-v\otimes v}{(1+|v|^2)^\frac{3}{2}}\Big)\Big|_Fdx +\frac{r_2}{2}\int_{\mathrm{\Omega}}\big(w-{\nabla}^2 u^{k+1}-\frac{\lambda_2^k}{r_2}\big)^2dx,
\end{split}
\right.
\end{eqnarray*}
and then update the multipliers $\lambda_1$, $\lambda_2$ by gradient ascent method.

\subsection{The solution to the $u$-subproblem}
Given the fixed variables $v^k,w^k,\lambda_1^k,\lambda_2^k$, we pursue the Euler-Lagrange equation of the $u$-subproblem as the following linear partial differential equation (PDE)
\begin{equation*}
\frac1{\lambda}(u^{k+1}-f)-r_1 \div \big(\nabla u^{k+1}-(v^k-\frac{\lambda_1^k}{r_1})\big)+r_2{\div}^2 \big({\nabla}^2 u^{k+1}-(w^k-\frac{\lambda_2^k}{r_2})\big)=0,
\end{equation*}
which can be simplified as
\begin{equation*}
\Big(\frac1{\lambda}-r_1\triangle+r_2\triangle^2 \Big)u^{k+1}=f/\lambda-\div(r_1v^{k}-\lambda_1^k) +{\div}^2(r_2w^{k}-\lambda_2^k)
\end{equation*}
with $\triangle^2=\div^2 {\nabla}^2$. As long as the periodic boundary condition is adopted, we can utilize FFT to achieve the optimal solution $u^{k+1}$ from
\begin{equation}\label{usolution}
u^{k+1} = \mathcal{F}^{-1} \bigg(\frac{\mathcal{F}\big(f/{\lambda}-\div(r_1v^k-\lambda_1^k)+{\div}^2  (r_2w^k-\lambda_2^k)\big)}{(1/{\lambda})\mathcal{I}-r_1\mathcal{F}\triangle\mathcal{F}^{-1}+r_2\mathcal{F}\triangle^2\mathcal{F}^{-1}} \bigg),
\end{equation}
where $\mathcal{F}$ and $\mathcal{F}^{-1}$ represent the commonly used forward and inverse FFT operation, respectively.

\subsection{The solution to the $v$-subproblem}
Because the $v$-subproblem is a nonlinear and non-convex minimization problem, we employ the gradient descent method to seek an approximated solution.
The Euler-Lagrange equation of the $v$-subproblem is given by
\begin{equation*}
{\rm\Phi}'(\Lambda)w^k\Big[-\frac{3}{(1+|v|^2)^\frac{3}{2}}({\mathrm{I}}-{\mathrm{\Psi}})v\Big]+r_1(v-\nabla u^{k+1})-\lambda^k_1=0, ~~\mbox{with}~~\Lambda=w^{k}\Big(\frac{(1+|v|^2)\mathcal{I}-v\otimes v}{(1+|v|^2)^\frac{3}{2}}\Big),
\end{equation*}
where  ${\rm\Phi}:\mathbb R^{2\times2}\rightarrow \mathbb R$ is defined as
\begin{eqnarray*}
{\rm\Phi}(x)=
\begin{cases}
\sqrt{x^2+\epsilon}, &x=0,\\
|x|_F, &x\neq 0,
\end{cases}
\end{eqnarray*}
with $\epsilon$ being a small positive constant, and $\mathrm{I}$, $\mathrm{\Psi}$: $\mathbb{R}^2\rightarrow \mathbb{R}^2$ are defined as $\mathrm{I}(x)=x$ and $\mathrm{\Psi}(x)=\frac{v\otimes v}{1+|v|^2}x$, respectively. Supposing that the periodic boundary condition is used, we can estimate $v^{k+1}$ according to the following fourth-order evolution equation with time as an evolution parameter
\begin{equation}\label{vsolution}
\frac{\partial v}{\partial t} = -{\rm\Phi}'(\Lambda)w^k\Big[-\frac{3}{(1+|v|^2)^\frac{3}{2}}(\mathrm{I}-\mathrm{\Psi})v\Big]+r_1(\nabla u^{k+1}-v)+\lambda^k_1.
\end{equation}

\subsection{The solution to the $w$-subproblem}
The $w$-subproblem is a typical $L_1$ minimization problem, which can be effectively solved using the shrinkage operator \cite{beck2009fast,bioucas2007new} as follows
\begin{equation}\label{wsolution}
 w^{k+1}={\mathrm{shrinkage}_F}\Big(\nabla^2 u^{k+1}+\frac{\lambda_2^k}{r_2},\frac{|(1+|v^{k+1}|^2)\mathcal{I}-v^{k+1}\otimes v^{k+1}|_F}{r_2(1+|v^{k+1}|^2)^\frac{3}{2}}\Big).
\end{equation}
Note that the shrinkage operator ${\mathrm{shrinkage}_F}(a,\xi)$ is implemented on each pixel over $\rm\Omega$ such as
\begin{eqnarray*}
{\mathrm{shrinkage}_F}(a,\xi)=\max\big(|a|_F-\xi,0\big)\frac{a}{|a|_F},
\end{eqnarray*}
where $a\in\mathbb R^{2\times2}$ is a $2\times 2$ matrix.

\subsection{Update of the Lagrange multipliers $(\lambda_1,\lambda_2)$}
The Lagrange multipliers $(\lambda_1,\lambda_2)$ are updated through a standard dual-ascent rule from
\begin{eqnarray}\label{lagranmul}
\begin{cases}
\lambda_1^{k+1}&=\lambda_1^k+r_1(\nabla u^{k+1}-v^{k+1}),\\
\lambda_2^{k+1}&=\lambda_2^k+r_2({\nabla}^2 u^{k+1}-w^{k+1}).
\end{cases}
\end{eqnarray}
Based on the above discussion on the solutions to each variable individually, we then summarize the iterative procedure for solving the Weingarten map minimization model \eqref{coupleformula} in Algorithm \ref{wmALM}.
\begin{algorithm}
    \caption{The ADMM for Weingarten map minimization model \eqref{wHM}}
    \begin{algorithmic}[1]\\
    \textbf{Input:} Degraded image $f$, regularization parameter $\lambda$, penalty factors $(r_1,r_2)$, time stepsize $\Delta t$, maximum iteration ${K_{\max}}$, and stopping threshold $\varepsilon$;\\
    \textbf{Initialize:} ${u^0} = f$ and $v^0 = w^0 = \lambda_1^0 = \lambda_2^0 = 0$, set $k=0$;
    \While {(not converged and $k \leq {K_{\max}}$)} \\
    ~~Compute $u^{k+1}$ from Eq. \eqref{usolution} for fixed $v^k$, $w^k$, $\lambda_1^k$ and $\lambda_2^k$;\\
    ~~Compute $v^{k+1}$ from Eq. \eqref{vsolution} for fixed $u^{k+1}$, $w^k$ and $\lambda_1^k$;\\
    ~~Compute $w^{k+1}$ from Eq. \eqref{wsolution} for fixed $u^{k+1}$, $v^{k+1}$ and $\lambda_2^k$;\\
    ~~Update $\lambda_1^{k+1}$ and $\lambda_2^{k+1}$ according to  \eqref{lagranmul};\\
    ~~Check convergence condition: $\|u^{k + 1}-u^k\|_{V,1}/|\mathrm{\Omega}| \leq \varepsilon;$
    \EndWhile
    \State \textbf{Output:} Reconstructed image $u$.
  \end{algorithmic}
  \label{wmALM}
\end{algorithm}

\section{Spatially adaptive first and second order regularization model}
\label{sec4}
In Algorithm 1, the $v$-subproblem needs to be solved by gradient descent, which converges slowly in practice. The numerical difficulty is mainly due to the coupling of the two terms in the Weingarten map. Thus, we consider minimizing the following energy functional
\begin{equation*}
F(u)=\int_{\mathrm{\Omega}}\Big|\nabla \frac{1}{\sqrt{1+|\nabla u|^2}}\otimes\nabla u\Big|_Fdx + \int_{\mathrm{\Omega}}\Big|\frac{1}{\sqrt{1+|\nabla u|^2}}\nabla^2 u\Big|_Fdx + \frac{1}{2\lambda}\int_{\mathrm{\Omega}}(u-f)^2dx.
\end{equation*}
Note that the following result can be obtained for Frobenius norm
\begin{lem}
If $a\in\mathbb R^{n}$ and $b\in\mathbb R^{n}$ are two vectors, then $|a\otimes b|_F = |a||b|$.
\end{lem}
Therefore, we can further reformulate the above functional into a hybrid nonlinear first and second order regularization problem as follows
\begin{equation}
\min_u~\int_{\mathrm{\Omega}}\alpha(u)|\nabla u|dx + \int_{\mathrm{\Omega}}\beta(u)|\nabla^2 u|_Fdx+\frac{1}{2\lambda}\int_{\mathrm{\Omega}}(u-f)^2dx,
\label{wtv2}
\end{equation}
with
\begin{align*}
\alpha(u)=\Big|\nabla \frac{1}{\sqrt{1+|\nabla u|^2}}\Big|,~~\mbox{and}~~\beta(u)=\frac{1}{\sqrt{1+|\nabla u|^2}},
\label{alpha_beta}
\end{align*}
where $\beta(u)$ is an edge detector function and $\alpha(u)$ is the total variation of $\beta(u)$.
As a matter of fact, similar or partial models have been studied in the literature. Chan, Marquina and Mulet \cite{chan2000high} introduced a nonlinear second order regularization to the TV functional for denoising problem. Specifically, the edge detection function $\beta(u)$ is coupled with the elliptic operator to eliminate the action of high order regularization on edges. Likewise, different edge detector functions have been introduced to first or second order variational models for various image processing problems \cite{bresson2007fast,Li2007,Zhang2013,Duan2016}. However, these works estimated the edge detector function using the observed images and treated it as the spatially adapted parameter for the regularization terms.
For example, Bresson \emph{et al.} \cite{bresson2007fast} introduced the edge detector function as the weights for TV term in image denoising. Li \emph{et al.} \cite{Li2007} developed a high order denoising model, where the edge detector function $g$ and $1-g$ were used as spatially adapted parameters for first and second order terms, respectively.

\subsection{The constrained optimization problem and ADMM-based algorithm}
The main computational challenges of the model \eqref{wtv2} come from the nonlinear terms $\alpha(u)$ and $\beta(u)$.  As explored for Euler's elastica model \cite{bae2011graph,yashtini2016fast}, the functional \eqref{wtv2} can be regarded as a weighted first and second order regularization model by computing $\alpha(u)$ and $\beta(u)$ separately in an iterative way. Then effective and efficient numerical algorithms can be used to solve the minimization problem such as augmented Lagrangian method \cite{wu2010augmented}, the split Bregman method \cite{goldstein2009split} and primal-dual splitting method \cite{Chambolle2010}, etc.

In particular, we introduce two auxiliary variables $v$ and $w$ and rewrite the original unconstrained optimization problem \eqref{wtv2} into a constrained version as follows
\begin{equation}\label{constrained version}
\begin{split}
& \min_{(u,v,w)\in V\times Q_1\times Q_2}~\int_{\mathrm{\Omega}}\alpha(x)|v|dx+\int_{\mathrm{\Omega}}\beta(x)|w|_Fdx+\frac{1}{2\lambda}\int_{\mathrm{\Omega}}(u-f)^2dx \\
& \qquad\quad\mathrm{s.t.}\qquad~~ v=\nabla u,~ w={\nabla}^2 u,
\end{split}
\end{equation}
where $\alpha(x)$ and $\beta(x)$ are evaluated in a separate step.
Given some $(u^k, v^k, w^k)\in V\times Q_1\times Q_2$, the augmented Lagrangian functional is defined as follows
\begin{align}\label{Lagrangian equation}
\begin{split}
 \mathcal{L}({u,v,w; \lambda_1,\lambda_2}) = &\int_{\mathrm{\Omega}}\alpha(x)|v|dx+\int_{\mathrm{\Omega}}\beta(x)|w|_Fdx+\frac{1}{2\lambda}\int_{\mathrm{\Omega}}(u-f)^2dx
 -\int_{\mathrm{\Omega}}\lambda_1(v-\nabla u) dx \\
 &+ \frac{r_1}{2}\int_{\mathrm{\Omega}}\big(v-\nabla u\big)^2dx - \int_{\mathrm{\Omega}}\lambda_2(w-{\nabla}^2 u)dx + \frac{r_2}{2}\int_{\mathrm{\Omega}}\big(w-{\nabla}^2 u\big)^2dx,\\
\end{split}
\end{align}
where $(\lambda_1$, $\lambda_2)\in Q_1\times Q_2$ are the Lagrange multipliers, and $r_1$, $r_2$ are the positive pently parameters. During each iteration, by the alternating direction method of multipliers, we tend to sequentially minimize \eqref{Lagrangian equation} over variables $(u,v,w)$ while keeping the reminder variables fixed. The minimizers $u^{k+1}$, $v^{k+1}$, $w^{k+1}$ are estimated from
\begin{eqnarray}\label{subproblems}
\left \{
\begin{split}
u^{k+1} &= \arg \min_{u\in V}~\frac{1}{2\lambda}\int_{\mathrm{\Omega}}(u-f)^2dx+\frac{r_1}{2}\int_{\mathrm{\Omega}}\big(\nabla u-(v^k-\frac{\lambda_1^k}{r_1})\big)^2dx+\frac{r_2}{2}\int_{\mathrm{\Omega}}\big({\nabla}^2 u-(w^k-\frac{\lambda_2^k}{r_2})\big)^2dx,\\
v^{k+1} &= \arg \min_{v\in Q_1}~\int_{\mathrm{\Omega}}\alpha(u^{k+1})|v|dx+\frac{r_1}{2}\int_{\mathrm{\Omega}}\big(v-\nabla u^{k+1}-\frac{\lambda_1^k}{r_1}\big)^2dx,\\
w^{k+1} &= \arg \min_{w\in Q_2}~\int_{\mathrm{\Omega}}\beta(u^{k+1})|w|_Fdx+\frac{r_2}{2}\int_{\mathrm{\Omega}}\big(w-{\nabla}^2 u^{k+1}-\frac{\lambda_2^k}{r_2}\big)^2dx,
\end{split}
\right.
\end{eqnarray}
and then the Lagrange multipliers $(\lambda_1,\lambda_2)$ are updated through a standard dual-ascent rule from
\begin{eqnarray}\label{lag_mul}
\begin{cases}
\lambda_1^{k+1}&=\lambda_1^k+r_1(\nabla u^{k+1}-v^{k+1}),\\
\lambda_2^{k+1}&=\lambda_2^k+r_2({\nabla}^2 u^{k+1}-w^{k+1}),
\end{cases}
\end{eqnarray}
where both $\alpha(u^{k+1})$ and $\beta(u^{k+1})$ are of known values as
\begin{equation}\label{updateparameters}
 \alpha(u^{k+1}) = \Big|\nabla {\frac{1}{\sqrt{1+|\nabla u^{k+1}|^2}}}\Big| ~~{\mathrm{and}}~~ \beta(u^{k+1}) = \frac{1}{\sqrt{1+|\nabla u^{k+1}|^2}}.
\end{equation}

\subsection{The solutions to subproblems}
\subsubsection{The sub-minimization problem w.r.t. $u$}
With the fixed variables $v^k,w^k,\lambda_1^k,\lambda_2^k$ at the $(k+1)$-th outer iteration, the Euler-Lagrange equation of the $u$-subproblem is given by
\begin{equation*}
\frac1{\lambda}(u^{k+1}-f)-r_1 \div \big(\nabla u^{k+1}-(v^k-\frac{\lambda_1^k}{r_1})\big)+r_2{\div}^2 \big({\nabla}^2 u^{k+1}-(w^k-\frac{\lambda_2^k}{r_2})\big)=0,
\end{equation*}
which can be simplified as
\begin{equation*}
\Big(\frac{1}{\lambda}-r_1\triangle+r_2\triangle^2\Big)u^{k+1}=f/\lambda-\div(r_1v^{k}-\lambda_1^k)+{\div}^2(r_2w^{k}-\lambda_2^k).
\end{equation*}
Suppose the periodic boundary condition is imposed, we can use the FFT to obtain the optimal solution $u^{k+1}$ from
\begin{equation}\label{u-subproblem}
u^{k+1} = \mathcal{F}^{-1} \bigg(\frac{\mathcal{F}\big(f/{\lambda}-\div(r_1v^k-\lambda_1^k)+{\div}^2  (r_2w^k-\lambda_2^k)\big)}{(1/{\lambda})\mathcal{I}-r_1\mathcal{F}\triangle\mathcal{F}^{-1}+r_2\mathcal{F}\triangle^2\mathcal{F}^{-1}} \bigg).
\end{equation}

\subsubsection{The sub-minimization problem w.r.t. $(v,w)$}
Both the $v$-subproblem and $w$-subproblem in \eqref{subproblems} are component-wise separable, which can be solved by shrinkage operators. To be specific, the solution to the variable $v$ is obtained by the isotropic shrinkage operator defined for vectors
\begin{equation}\label{v-subproblem}
v^{k+1}= {\mathrm{shrinkage}_2}\bigg( \nabla u^{k+1}+\frac{\lambda_1^k}{r_1},\frac{\alpha(u^{k+1})}{r_1} \bigg)
\end{equation}
with
\[{\mathrm{shrinkage}_2}(b,\xi)={\mathrm{max}}\{|b|-\xi,0\} \frac{b}{|b|},\quad \mbox{for}~ b\in \mathbb R^n.\]
Likewise, we have the solution to the $w$-subproblem as follows
\begin{equation}\label{w-subproblem}
w^{k+1}= {\mathrm{shrinkage}_F}\bigg( {\nabla}^2 u^{k+1}+\frac{\lambda_2^k}{r_2},\frac{\beta(u^{k+1})}{r_2} \bigg).
\end{equation}

In brief, an efficient ADMM-based numerical algorithm is proposed to deal with the spatially adapted first and second order regularization model \eqref{wtv2}, the optimization procedure of which is sketched in Algorithm 2.

\begin{algorithm}
    \caption{The ADMM for spatially adapted first and second order regularization model \eqref{wtv2}}
    \begin{algorithmic}[1]\\
    \textbf{Input:} Degraded image $f$, positive constant $\lambda$, penalty factors $(r_1,r_2)$, maximum iteration ${K_{\max }}$, and stopping threshold $\varepsilon$;\\
    \textbf{Initialize:} ${u^0} = f$ and $v^0 = w^0 = \lambda_1^0 = \lambda_2^0 = 0$, set $k=0$;
    \While {(not converged and $k \leq {K_{\max}}$)} \\
    ~~Compute $u^{k+1}$ from Eq. \eqref{u-subproblem} for fixed $v^k$, $w^k$, $\lambda_1^k$ and $\lambda_2^k$;\\
    ~~Update $\alpha(u^{k+1})$ and $\beta(u^{k+1})$ using $u^{k+1}$ according to Eq. \eqref{updateparameters};\\
    ~~Compute $v^{k+1}$ from Eq. \eqref{v-subproblem} for fixed $u^{k+1}$ and $\lambda_1^k$;\\
    ~~Compute $w^{k+1}$ from Eq. \eqref{w-subproblem} for fixed $u^{k+1}$ and $\lambda_2^k$;\\
    ~~Update $\lambda_1^{k+1}$ and $\lambda_2^{k+1}$ according to  \eqref{lag_mul};\\
    ~~Check convergence condition: $\|u^{k + 1}-u^k\|_{V,1}/|\mathrm{\Omega}| \leq \varepsilon;$
    \EndWhile
    \State \textbf{Output:} Reconstructed image $u$.
  \end{algorithmic}
  \label{HSAVM}
\end{algorithm}

\section{Experimental results}
\label{sec5}
In this section, comprehensive experiments consisting of three parts, i.e., image denoising, image deblurring and image inpainting are implemented to verify the efficiency and superiority of the proposed Algorithm 1 for Weingarten map regularization model (denoted by WM) and Algorithm 2 for the reformulated spatially adapted first and second order variational model (denoted by SA-TV-TV$^2$). All numerical experiments are performed utilizing Matlab R2016a on a machine with 3.40GHz Intel(R) Core(TM) i7-6700 CPU and 32GB RAM.

In our work, the popular peak signal-to-noise ratio (PSNR) and structural similarity (SSIM) indexes \cite{wang2004image} are adopted to quantitatively evaluate the imaging performance under different image degradation conditions. In particular, the PSNR is defined as
\begin{equation}\label{psnr}
 {\mathrm {PSNR}}(u_0,u) = 10 {\mathrm {log}}{\frac{255^2}{\mathrm{MSE}}},
\end{equation}
and the SSIM is given as
\begin{equation}\label{ssim}
 {\mathrm {SSIM}}(u_0,u) = \frac{(2\mu_{u_0}\mu_u+c_1)(2\sigma_{u_0 u}+c_2)}{({\mu_{u_0}}^2+{\mu_{u}}^2+c_1)({\sigma_{u_0}}^2+{\sigma_{u}}^2+c_2)},
\end{equation}
where $u_0$ denotes the clear image, $u$ represents the recovery image, MSE indicates the mean square error of $u_0$ and $u$. The ${\mu_{u_0}}$ and ${\mu_{u}}$ express the local mean values of images $u_0$ and $u$, ${\sigma_{u_0}}$ and ${\sigma_{u}}$ signify the respective standard deviations, ${c_1}$ and ${c_2}$ are two constants to avoid instability for near zero denominator values, and ${\sigma_{u_0 u}}$ is the covariance value between images $u_0$ and $u$. Theoretically, higher PSNR and SSIM values normally indicate better performance in image reconstruction.

\begin{figure*}[t]
      \begin{center}
			\includegraphics[width=1.00\linewidth]{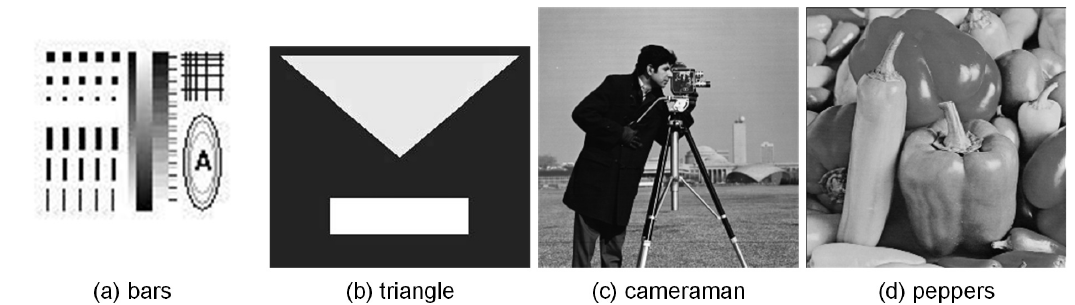}
	  \end{center}
	\caption{Test images. From left to right: (a) bars($128\times128$), (b) triangle($254\times214$), (c) cameraman($256\times256$), (d) peppers($256\times256$), respectively.}
	\label{testimages}
\end{figure*}

\begin{table*}[t]
	\centering
	\caption{\label{Parameters}The tunable parameters of comparative image reconstruction methods. Here, $\lambda^0$ indicates the initial value of $\lambda$, which is dynamically updated for the SATV model. }
	\begin{tabular}{c|c|c|c|c|c|c|c}
		\hline\hline
		 Methods & \multicolumn{3}{|c}{Model parameters} & \multicolumn{4}{|c}{Algorithm parameters} \\
        \hline
		 Euler's elastica \cite{tai2011fast} & $\eta$ & $a$ & $b$ & $r_1$ & $r_2$ & $-$ & $r_4$ \\ \hline
		 MC \cite{zhu2013augmented} & $\lambda$ & $-$ & $-$ & $r_1$ & $r_2$ & $r_3$ & $r_4$ \\ \hline
         TV-TV$^2$ \cite{papafitsoros2014combined} & $-$ & $\alpha$ & $\beta$ & $r_1$ & $r_2$ & $-$ & $-$ \\ \hline
         TGV \cite{bredies2010total} & $\lambda$ & $\alpha_0$ & $\alpha_1$ & $r_1$ & $r_2$ & $-$ & $-$ \\ \hline
		 SATV \cite{dong2011automated} & $\lambda^0$ & $-$ & $-$ & $\omega$ & $\zeta$ & $-$ & $-$ \\ \hline
		 WM & $\lambda$ & $-$ & $-$ & $r_1$ & $r_2$ & $-$ & $-$ \\ \hline
		 SA-TV-TV$^2$ & $\lambda$ & $-$ & $-$ & $r_1$ & $r_2$ & $-$ & $-$ \\
		\hline\hline
	\end{tabular}
\end{table*}

The variation of the relative residuals, the relative errors and numerical energy can provide important information about the numerical convergence of the proposed Algorithm 1 and Algorithm 2. Therefore, we track the relative residuals during the iterations, which is defined as
\begin{equation}\label{residual}
 (R_1^k,R_2^k) = \frac{1}{|\mathrm{\Omega}|}(\|v^k-\nabla u^k\|_{Q_1,1},\|w^k-{\nabla}^2 u^k\|_{Q_2,1}),
\end{equation}
where $\|\cdot\|_{Q_1,1}$, $\|\cdot\|_{Q_2,1}$ denote the $L^1$ norm in $Q_1$ and $Q_2$, respectively, and $|\mathrm{\Omega}|$ is the area of the image domain. Simultaneously, we check the relative errors of the Lagrange multipliers
\begin{equation}\label{multiplier}
 (L_1^k,L_2^k) = \frac{1}{|\mathrm{\Omega}|}(\|\lambda_1^k-\lambda_1^{k-1}\|_{Q_1,1},\|\lambda_2^k-\lambda_2^{k-1}\|_{Q_2,1})
\end{equation}
and the relative error in $u^k$
\begin{equation}\label{error}
R(u^k)=\frac{\|u^k-u^{k-1}\|_{V,1}}{|\mathrm{\Omega}|},
\end{equation}
where $\|\cdot\|_{V,1}$ is the $L^1$ norm defined in $V$. Besides, the numerical energy is calculated by
\begin{equation}\label{energy1}
E(u^k) =
\int_{\mathrm{\Omega}}\Big|\nabla^2 u^k\Big( \frac{(1+|\nabla u^k|^2)\mathcal{I}-\nabla u^k\otimes \nabla u^k}{(1+|\nabla u^k|^2)^\frac{3}{2}}\Big)\Big|_Fdx + \frac{1}{2\lambda}\int_{\mathrm{\Omega}}(u^k-f)^2dx
\end{equation}
for the Weingarten map minimization model \eqref{wHM}, and
\begin{equation}
F(u^k) =\int_{\mathrm{\Omega}}\alpha(u^k)|\nabla u^k|dx+\int_{\mathrm{\Omega}}\beta(u^k)|{\nabla}^2 u^k|_Fdx+\frac{1}{2\lambda}\int_{\mathrm{\Omega}}(u^k-f)^2dx
\label{energy2}
\end{equation}
for the spatially adapted first and second order regularization model \eqref{wtv2}, respectively.

\begin{figure*}[t]
      \begin{center}
			\includegraphics[width=1.00\linewidth]{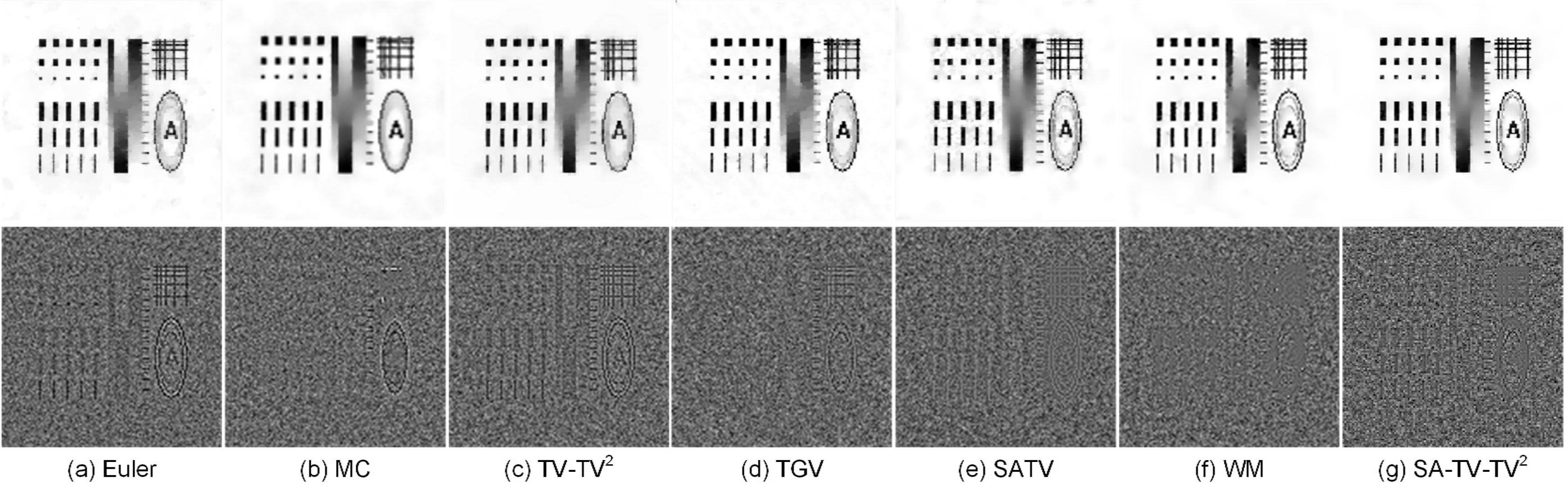}
	  \end{center}
	\caption{Denoising results of ``bars" (top) and their residual images (bottom) by different methods. The parameters are set as (a) Euler's elastica: $a=1$, $b=10$, $\eta=1.5\cdot10^2$, $r_1=1$, $r_2=2\cdot10^2$ and $r_4=5\cdot10^2$; (b) MC: $r_1=20$, $r_2=20$, $r_3=10^4$, $r_4=10^5$ and $\lambda=1.5\cdot10^3$; (c) TV-TV$^2$: $\alpha=10$, $\beta=5$, $r_1=1$ and $r_2=5$; (d) TGV: $\alpha_0=1.5$, $\alpha_1=1.0$, $r_1=10$, $r_2=50$ and $\lambda=4$; (e) SATV: $\omega=11$, $\zeta=2$ and $\lambda^0=2.0$; (f) WM: $r_1=0.1$, $r_2=0.5$, $\Delta t=0.1$ and $\lambda=200$; (g) SA-TV-TV$^2$: $r_1=0.1$, $r_2=0.5$ and $\lambda=160$.}
	\label{bars}
\end{figure*}

\begin{figure*}[t]
      \begin{center}
			\includegraphics[width=1.00\linewidth]{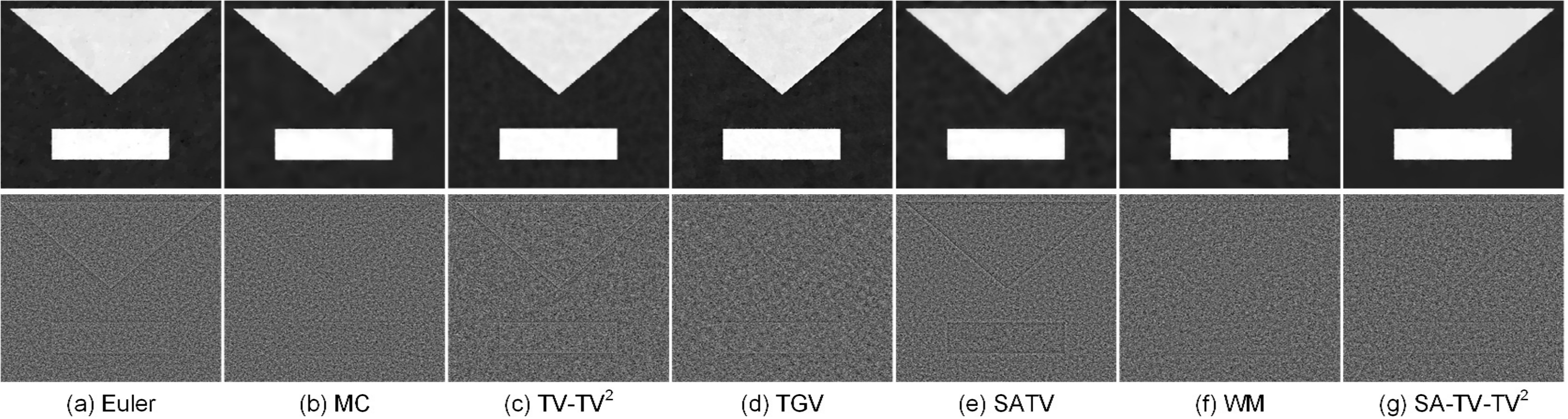}
	  \end{center}
	\caption{Denoising results of ``triangle" (top) and their residual images (bottom) by different methods. The parameters are set as (a) Euler's elastica: $a=1$, $b=10$, $\eta=1.5\cdot10^2$, $r_1=1$, $r_2=2\cdot10^2$ and $r_4=5\cdot10^2$; (b) MC: $r_1=20$, $r_2=20$, $r_3=10^5$, $r_4=10^5$ and $\lambda=1.5\cdot10^3$; (c) TV-TV$^2$: $\alpha=10$, $\beta=5$, $r_1=1$ and $r_2=5$; (d) TGV: $\alpha_0=1.5$, $\alpha_1=1.0$, $r_1=10$, $r_2=50$ and $\lambda=4$; (e) SATV: $\omega=11$, $\zeta=2$ and $\lambda^0=2.0$; (f) WM: $r_1=0.1$, $r_2=0.5$, $\Delta t=0.1$ and $\lambda=200$; (g) SA-TV-TV$^2$: $r_1=0.1$, $r_2=0.5$ and $\lambda=160$.}
	\label{triangle}
\end{figure*}

\begin{figure*}[t]
      \begin{center}
			\includegraphics[width=1.00\linewidth]{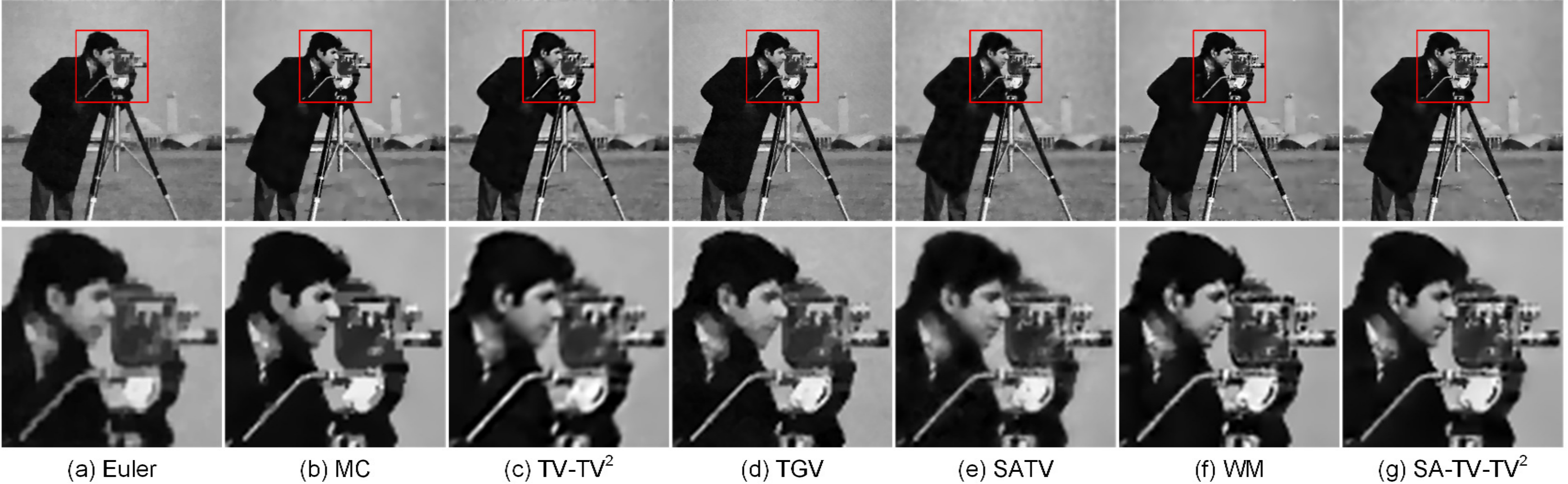}
	  \end{center}
	\caption{Denoising results of ``cameraman" (top) and their local magnification views (bottom) by different methods. The parameters are set as (a) Euler's elastica: $a=1$, $b=10$, $\eta=2\cdot10^2$, $r_1=1$, $r_2=2\cdot10^2$ and $r_4=5\cdot10^2$; (b) MC: $r_1=40$, $r_2=40$, $r_3=10^5$, $r_4=1.5\cdot10^5$ and $\lambda=10^2$; (c) TV-TV$^2$: $\alpha=4$, $\beta=8$, $r_1=10$ and $r_2=10$; (d) TGV: $\alpha_0=1.5$, $\alpha_1=1.0$, $r_1=10$, $r_2=50$ and $\lambda=10$; (e) SATV: $\omega=11$, $\zeta=2$ and $\lambda^0=2.5$; (f) WM: $r_1=1$, $r_2=2$, $\Delta t=0.01$ and $\lambda=90$; (g) SA-TV-TV$^2$: $r_1=1$, $r_2=2$ and $\lambda=100$.}
	\label{cameraman}
\end{figure*}

\begin{figure*}[t]
\centering
      \subfigure{
			\includegraphics[width=0.24\linewidth]{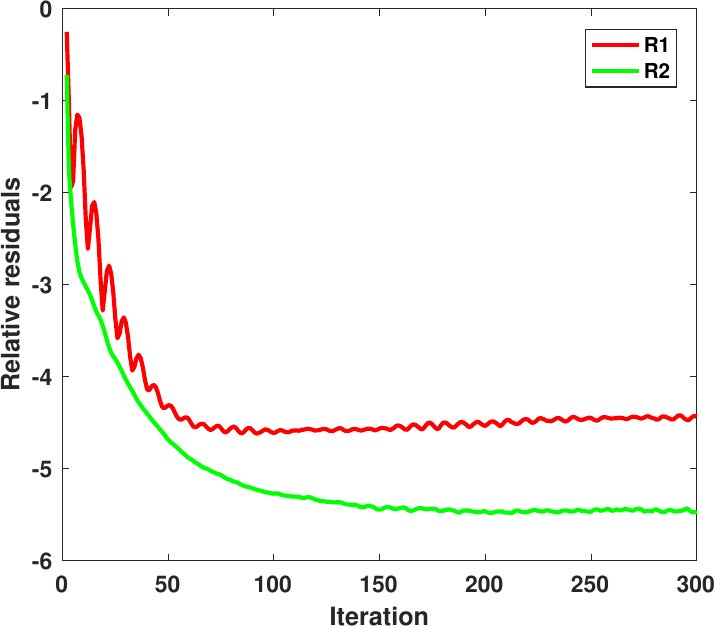}}\hspace{-0.5ex}
      \subfigure{
            \includegraphics[width=0.24\linewidth]{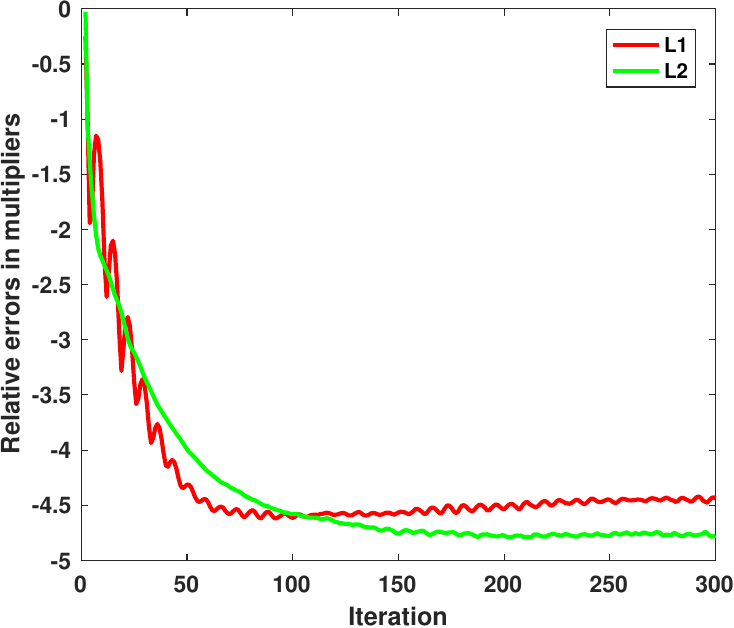}}\hspace{-0.5ex}
      \subfigure{
			\includegraphics[width=0.24\linewidth]{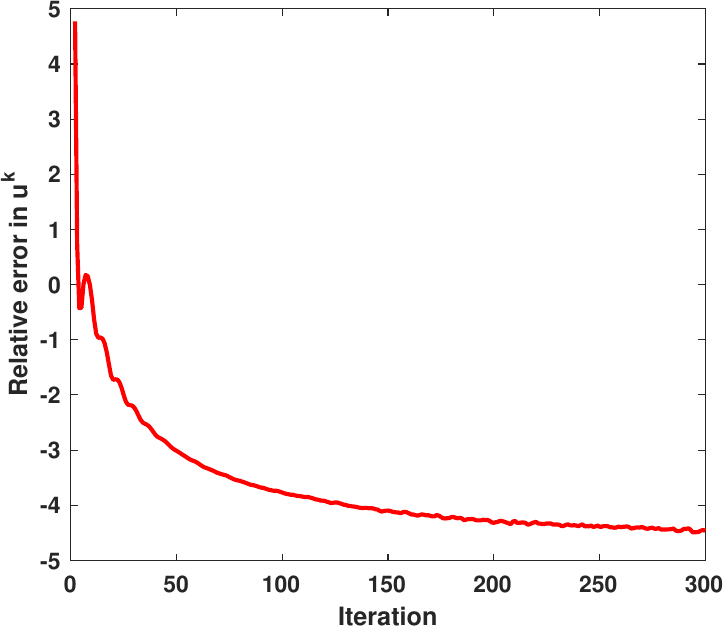}}\hspace{-0.5ex}
      \subfigure{
            \includegraphics[width=0.24\linewidth]{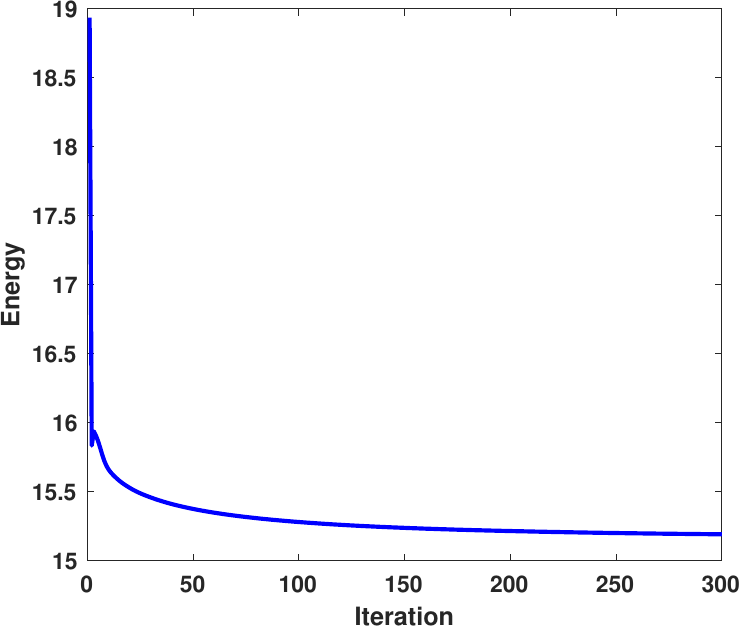}}
      \subfigure{
			\includegraphics[width=0.24\linewidth]{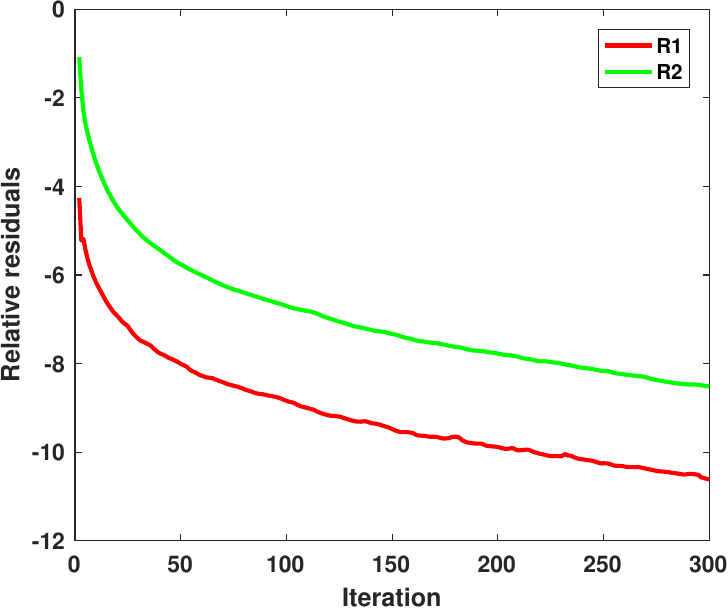}}\hspace{-0.5ex}
      \subfigure{
            \includegraphics[width=0.24\linewidth]{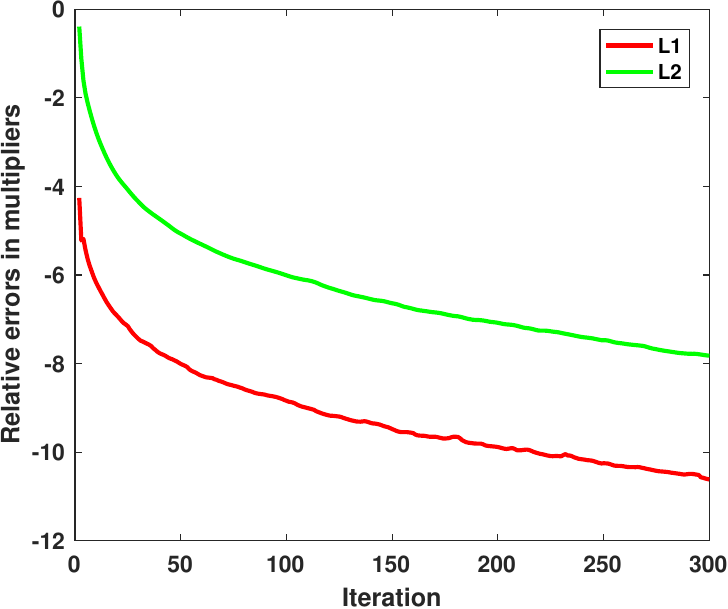}}\hspace{-0.5ex}
      \subfigure{
			\includegraphics[width=0.24\linewidth]{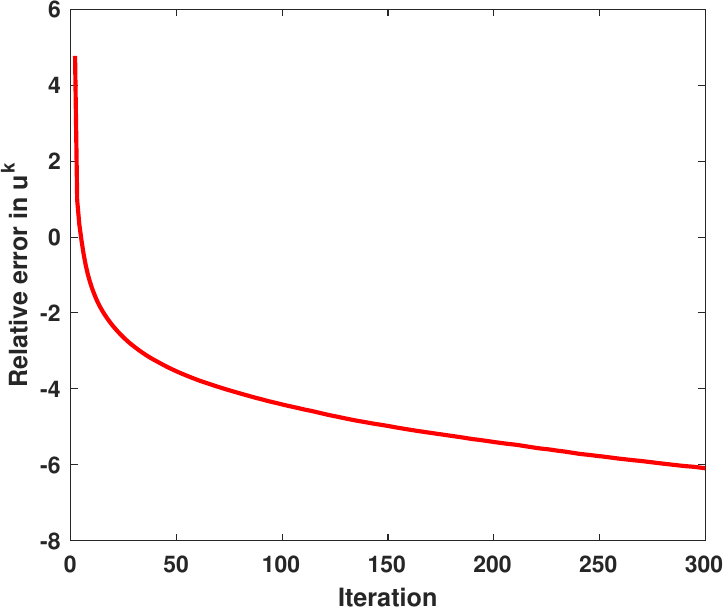}}\hspace{-0.5ex}
      \subfigure{
            \includegraphics[width=0.24\linewidth]{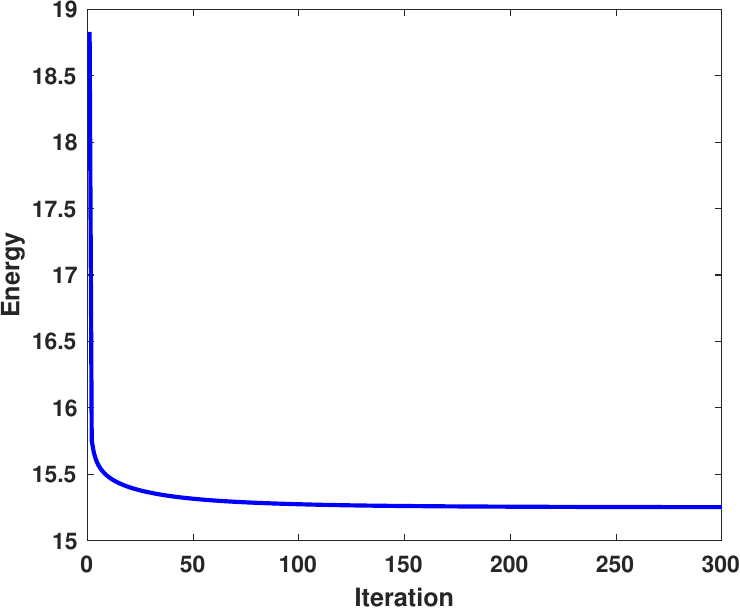}}
	\caption{Evaluations of ``cameraman" by the WM (top) and SA-TV-TV$^2$ (bottom) methods. From left to right: Relative residuals \eqref{residual}, relative errors in multipliers \eqref{multiplier}, relative errors in $u^k$ \eqref{error} and numerical energy \eqref{energy1} (top) and \eqref{energy2} (bottom), respectively.}
	\label{cameraman20}
\end{figure*}

\subsection{Numerical discretization}
Let ${\rm{\Omega}}=\{(i,j):0\leq i\leq m,0\leq j\leq n\}$ be the discretized image domain and $u(i,j)$ denote the intensity value of image $u$ at the pixel $(i,j)\in \rm \Omega$. We define the discrete forward $(+)$ and backward $(-)$ differential operators under periodic boundary condition as: $\partial_x^+u(i,j)=(u(i+1,j)-u(i,j))/\Delta x$, $\partial_y^+u(i,j)=(u(i,j+1)-u(i,j))/\Delta y$, $\partial_x^-u(i,j)=(u(i,j)-u(i-1,j))/\Delta x$, $\partial_y^-u(i,j)=(u(i,j)-u(i,j-1))/\Delta y$, where $\Delta x$ and $\Delta y$ denote the spatial mesh sizes. Then the discrete gradient operator $\nabla$: $\mathbb R^{m\times n}\rightarrow (\mathbb R^{m\times n})^2$ is given by $\nabla u(i,j) = \big(\partial_x^{+}u(i,j),\partial_y^{+}u(i,j)\big)$, and the discrete divergence operator div: $(\mathbb R^{m\times n})^2 \rightarrow \mathbb R^{m\times n}$ for $p=(p_1,p_2)\in(\mathbb R^{m\times n})^2$ is denoted as $\mathrm{div}p(i,j)=\partial_x^{-}p_1(i,j)+\partial_y^{-}p_2(i,j)$.

Correspondingly, based on periodic boundary condition, the discrete second order differential operators are further defined as
$\partial_{xx}^{-+}u(i,j)=\partial_{xx}^{+-}u(i,j)=\partial_x^-(\partial_x^+u(i,j)),~ \partial_{xy}^{++}u(i,j)=\partial_{yx}^{++}u(i,j)=\partial_x^+(\partial_y^+u(i,j)),~\partial_{xy}^{--}u(i,j)=\partial_{yx}^{--}u(i,j)=\partial_x^-(\partial_y^-u(i,j)),~
\partial_{yy}^{-+}u(i,j)=\partial_{yy}^{+-}u(i,j)=\partial_y^-(\partial_y^+u(i,j))$.
Therefore, the discrete Hessian operator $\nabla^2$: $\mathbb R^{m\times n}\rightarrow (\mathbb R^{m\times n})^4$ is denoted as
\[\nabla^2 u(i,j) = \begin{pmatrix} \partial_{xx}^{-+}u(i,j) & \partial_{xy}^{++}u(i,j) \\ \partial_{yx}^{++}u(i,j) & \partial_{yy}^{-+}u(i,j)\end{pmatrix}.\]
For $q=(q_{11},q_{12},q_{21},q_{22})\in(\mathbb R^{m\times n})^4$, the discrete second order divergence operator div$^2$: $(\mathbb R^{m\times n})^4 \rightarrow \mathbb R^{m\times n}$ is defined by
\begin{equation*}
\mathrm{div}^2q(i,j)=\partial_{xx}^{+-}q_{11}(i,j)+\partial_{xy}^{--}q_{12}(i,j)+\partial_{yx}^{--}q_{21}(i,j)+\partial_{yy}^{+-}q_{22}(i,j).
\end{equation*}

\begin{figure*}[t]
      \begin{center}
			\includegraphics[width=1.00\linewidth]{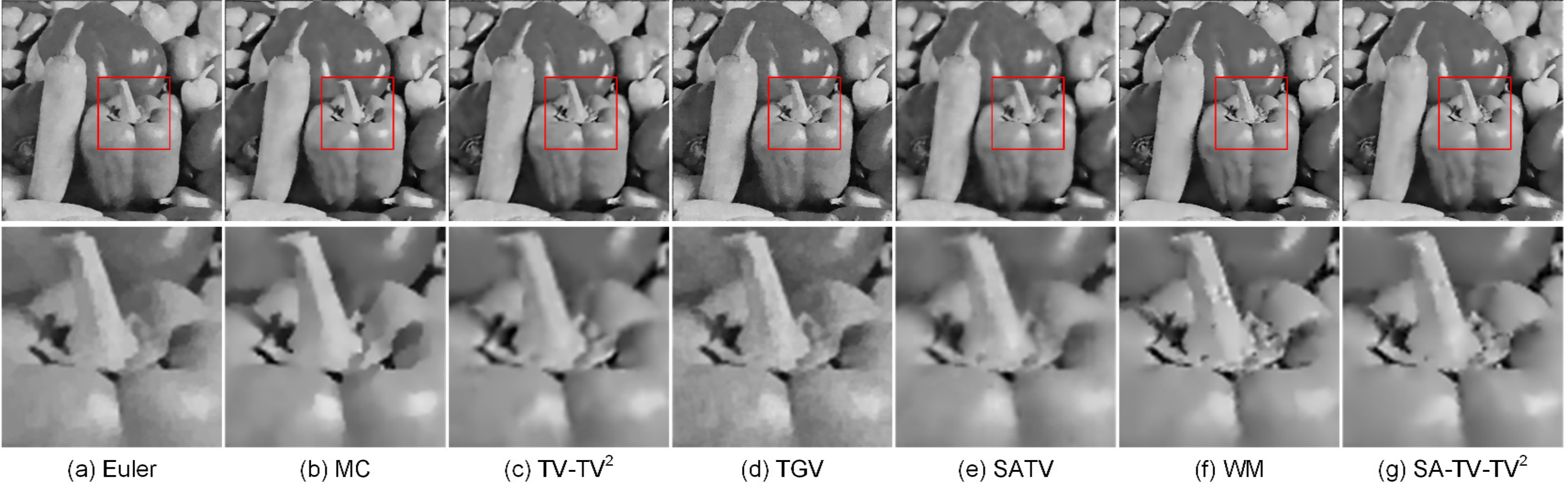}
	  \end{center}
	\caption{Denoising results of ``peppers" (top) and their local magnification views (bottom) by different methods. The parameters are set as (a) Euler's elastica: $a=1$, $b=10$, $\eta=2\cdot10^2$, $r_1=1$, $r_2=2\cdot10^2$ and $r_4=5\cdot10^2$; (b) MC: $r_1=40$, $r_2=40$, $r_3=10^5$, $r_4=10^5$ and $\lambda=10^2$; (c) TV-TV$^2$: $\alpha=4$, $\beta=10$, $r_1=10$ and $r_2=10$; (d) TGV: $\alpha_0=1.5$, $\alpha_1=1.0$, $r_1=10$, $r_2=50$ and $\lambda=10$; (e) SATV: $\omega=11$, $\zeta=2$ and $\lambda^0=2.5$; (f) WM: $r_1=1$, $r_2=2$, $\Delta t=0.01$ and $\lambda=90$; (g) SA-TV-TV$^2$: $r_1=0.1$, $r_2=0.5$ and $\lambda=100$.}
	\label{peppers}
\end{figure*}

\begin{table*}[t]
	\centering
	\caption{PSNR comparisons of various image denoising methods on test images for restoring noisy images corrupted by Gaussian noise with different standard deviation $\sigma$.}
	\begin{tabular}{c|c|c|c|c}
		\hline\hline
		 Methods & bars & triangle & cameraman & peppers \\
        \hline
		 Euler's elastica \cite{tai2011fast} & $24.83$ & $33.27$ & $28.05$ & $28.72$ \\ \hline
		 MC \cite{zhu2013augmented} & $26.35$ & $33.65$ & $28.40$ & $29.57$ \\ \hline
         TV-TV$^2$ \cite{papafitsoros2014combined} & $25.42$ & $33.06$ & $28.14$ & $28.50$ \\ \hline
         TGV \cite{bredies2010total} & $25.90$ & $33.41$ & $28.30$ & $29.17$ \\ \hline
		 SATV \cite{dong2011automated} & $25.64$ & $32.25$ & $28.28$ & $29.43$ \\ \hline
         WM & $27.01$ & $34.14$ & $29.13$ & $30.12$ \\ \hline
		 SA-TV-TV$^2$ & $\textbf{27.08}$ & $\textbf{34.20}$ & $\textbf{29.15}$ & $\textbf{30.28}$ \\
		\hline\hline
	\end{tabular}
\label{PSNR}
\end{table*}

\subsection{Parameters discussing and comparison methods}
There are three consistent parameters in the proposed Algorithm 1 and Algorithm 2, i.e., $\lambda$, $r_1$ and $r_2$. The regularization parameter $\lambda$ affects the contributions of the data-fidelity and regularization term, which should be selected according to the structures of the images and noise levels. The penalty parameters $r_1$ and $r_2$ control the convergent speed and stability of algorithms. To be specific, too small values of $r_1$ and $r_2$ usually reduce the algorithm's efficiency and relatively large values of $r_1$ and $r_2$ yield faster convergence. It is crucial to select appropriate penalty parameters $r_1$ and $r_2$ for balancing both algorithm's efficiency and stability. The time step size $\Delta t$ in Algorithm 1 is chosen as either $\Delta t = 0.1$ or $\Delta t = 0.01$ in different experiments. Similar to the mean curvature regularization \cite{zhu2013augmented}, the choice of spatial mesh sizes influences the reconstruction performance, which are set as $\Delta x =\Delta y = 5$ in the following experiments. 

We compare the proposed models with the most relevant methods including the Euler's elastica model (Euler) \cite{tai2011fast}, mean curvature (MC) \cite{zhu2013augmented}, hybrid first and second order model (TV-TV$^2$) \cite{papafitsoros2014combined}, the second order total generalized variation model (TGV) \cite{bredies2010total} and the spatially adapted TV method (SATV) \cite{dong2011automated}. The tunable parameters contained in different algorithms for comparison are listed in Table \ref{Parameters}. As can be seen, the proposed two methods not only contain fewer parameters, but also fewer subproblems in each iteration process.

\begin{figure}[t]
\centering
      \subfigure[tiangle]{
      \includegraphics[width=0.24\linewidth]{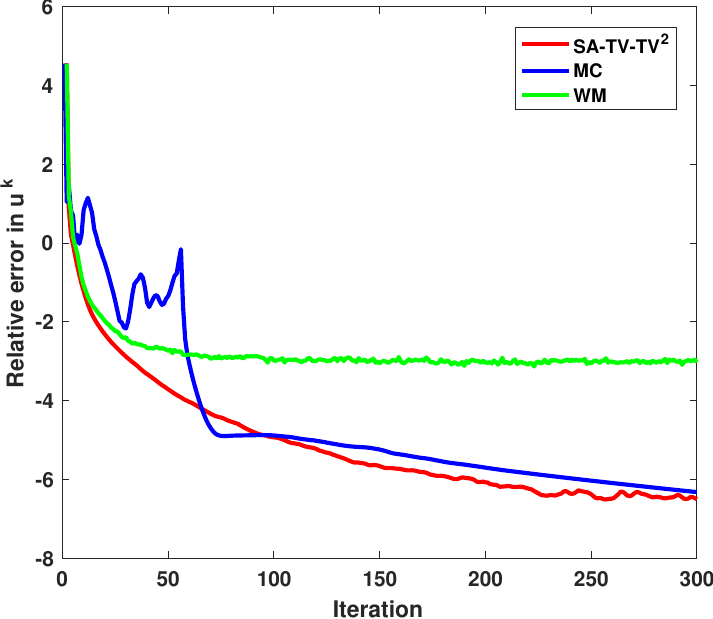}\hspace{0.5ex}
      \includegraphics[width=0.24\linewidth]{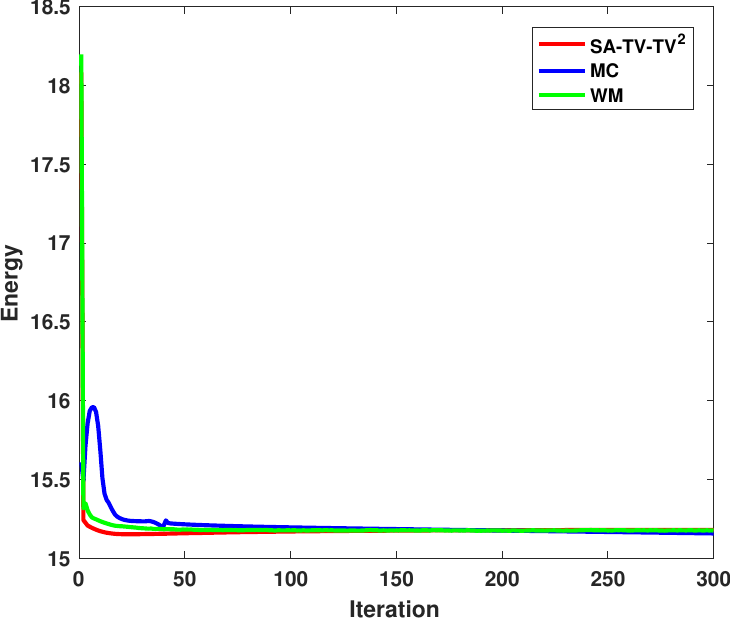}}
      \subfigure[peppers]{
	  \includegraphics[width=0.24\linewidth]{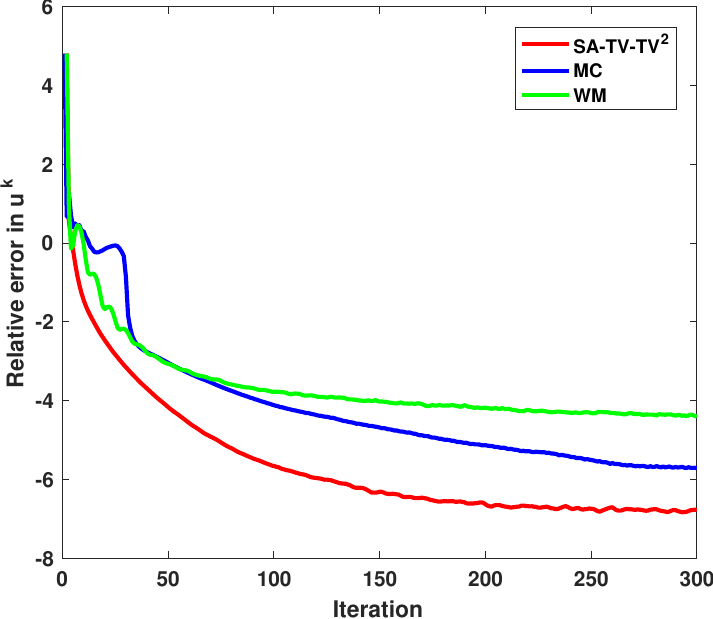}\hspace{0.5ex}
      \includegraphics[width=0.24\linewidth]{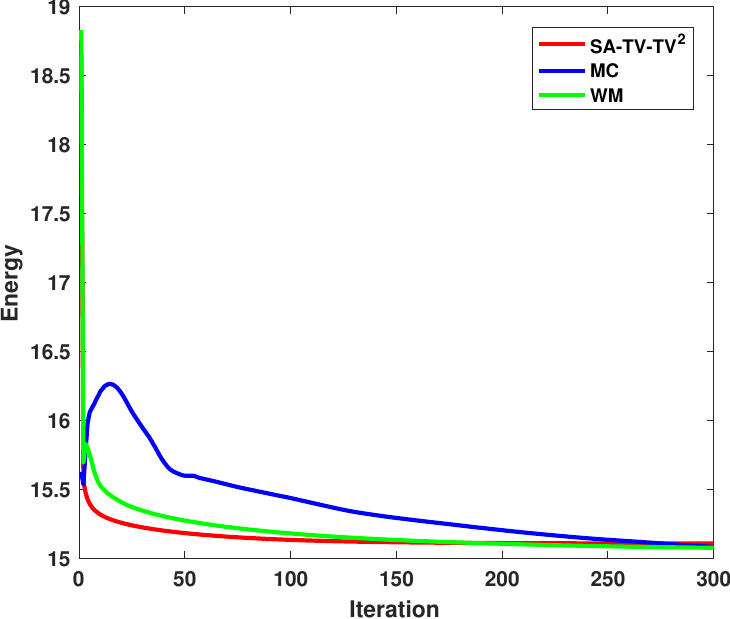}}
	\caption{Compared evaluations of ``triangle" and ``peppers'' by MC, WM and SA-TV-TV$^2$ methods in terms of relative errors and numerical energy.}
	\label{triangle30}
\end{figure}

\begin{table*}[t]
	\centering
	\caption{SSIM comparisons of various image denoising methods on test images for restoring noisy images corrupted by Gaussian noise with different standard deviation $\sigma$.}
	\begin{tabular}{c|c|c|c|c}
		\hline\hline
		 Methods & bars & triangle & cameraman & peppers \\
        \hline
		 Euler's elastica \cite{tai2011fast} & $0.9232$ & $0.9503$ & $0.8178$ & $0.8469$ \\ \hline
		 MC \cite{zhu2013augmented} & $0.9453$ & $0.9663$ & $0.8203$ & $0.8628$ \\ \hline
         TV-TV$^2$ \cite{papafitsoros2014combined} & $0.9308$ & $0.9561$ & $0.8224$ & $0.8571$ \\ \hline
         TGV \cite{bredies2010total} & $0.9380$ & $0.9358$ & $0.8157$ & $0.8474$ \\ \hline
		 SATV \cite{dong2011automated} & $0.9354$ & $0.9517$ & $0.8229$ & $0.8613$ \\ \hline
         WM  & $0.9554$ & $0.9715$ & $\textbf{0.8330}$ & $0.8762$ \\ \hline
		 SA-TV-TV$^2$ & $\textbf{0.9576}$ & $\textbf{0.9734}$ & $0.8284$ & $\textbf{0.8784}$ \\
		\hline\hline
	\end{tabular}\label{SSIM}
\end{table*}

\begin{table*}[t]
	\centering
	\caption{CPU time comparisons between various image denoising methods on test images corrupted by Gaussian noises, where the best two results are highlighted in bold and with underline, respectively.}
	\begin{tabular}{c|c|c|c|c}
		\hline\hline
		 Methods & bars & triangle & cameraman & peppers \\
        \hline
		 Euler's elastica \cite{tai2011fast} & $\textbf{2.48}$ & $\textbf{7.52}$ & $\textbf{8.69}$ & $\textbf{8.51}$ \\ \hline
		 MC \cite{zhu2013augmented} & $13.27$ & $41.02$ & $42.88$ & $42.41$ \\ \hline
         TV-TV$^2$ \cite{papafitsoros2014combined} & $7.94$ & $21.06$ & $22.58$ & $22.15$ \\ \hline
         TGV \cite{bredies2010total} & $10.09$ & $30.18$ & $31.98$ & $31.32$ \\ \hline
		 SATV \cite{dong2011automated} & $24.54$ & $98.85$ & $105.93$ & $103.14$ \\ \hline
		 WM  & $38.34$ & $117.77$ & $143.28$ & $142.53$ \\ \hline
		 SA-TV-TV$^2$ & $\underline{5.67}$ & $\underline{15.41}$ & $\underline{16.42}$ & $\underline{16.25}$ \\
		\hline\hline
	\end{tabular}\label{TIME}
\end{table*}

\subsection{Comparison experiments on image denoising}
We first illustrate the efficiency and superiority of the proposed models via various examples on image denoising. Four grayscale images displayed in Fig. \ref{testimages} are used to evaluate the performance of both our algorithms and comparison algorithms. To be specific, the synthetic images ``bars" and ``triangle" are degraded by Gaussian noises with zero mean and the standard deviation $\sigma=30$, while the real images ``cameraman" and ``peppers" are degraded by Gaussian noises with zero mean and the standard deviation $\sigma=20$. Different algorithms are stopped with the same termination condition such as $T_{\mathrm max}=300$ and $\varepsilon=2\times10^{-3}$ throughout this experiment. The specific values of both model and algorithm parameters for all comparison algorithms are provided separately in each example.

We display both the restoration results and the residual images of the two synthetic images in Fig. \ref{bars} and Fig. \ref{triangle}, and the denoising results and the selected local magnification views of the two real images in Fig. \ref{cameraman} and Fig. \ref{peppers}. In general, all methods can efficiently eliminate the noises, but only MC and our WM and SA-TV-TV$^2$ can well preserve the image structures and features. More specifically, the residual images obtained by the Euler's elastica, TV-TV$^2$ and SATV models contain many image details, while there is almost no signal left in the residual images of MC and WM, which confirms the contrast-preserving property of both MC and WM models. On the other hand, from the magnified images, we observe the TV-TV$^2$ model tends to obtain over-smoothed recovery results with blurry edges and missing details. The restored images of the Euler's elastica and TGV methods are not as smooth as others in the homogeneous regions. Although the SATV method can achieve almost satisfactory visual results owing to the spatially adapted regularization parameter, it suffers from some unnatural staircase-like artifacts in large homogeneous regions, e.g., the sky region in Fig. \ref{cameraman}. By contrast, the MC, WM and SA-TV-TV$^2$ models retain sharp edges and smoothed flat regions. The advantages of the contrast-preserving methods are also demonstrated by the PSNR and SSIM listed in Tables \ref{PSNR} and \ref{SSIM}. Although the MC model can also preserve image contrast, higher PSNR and SSIM are always achieved by our WM and SA-TV-TV$^2$ models owing to excellent geometric properties and spatially adapted operators.

\begin{figure*}[t]
      \centering
      \subfigure[triangle]{
			\includegraphics[width=0.35\linewidth]{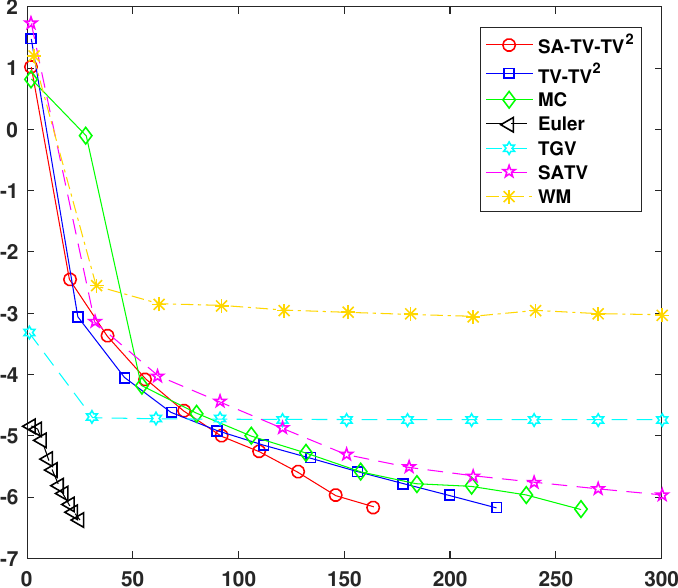}}\hspace{15ex}
      \subfigure[peppers]{
            \includegraphics[width=0.35\linewidth]{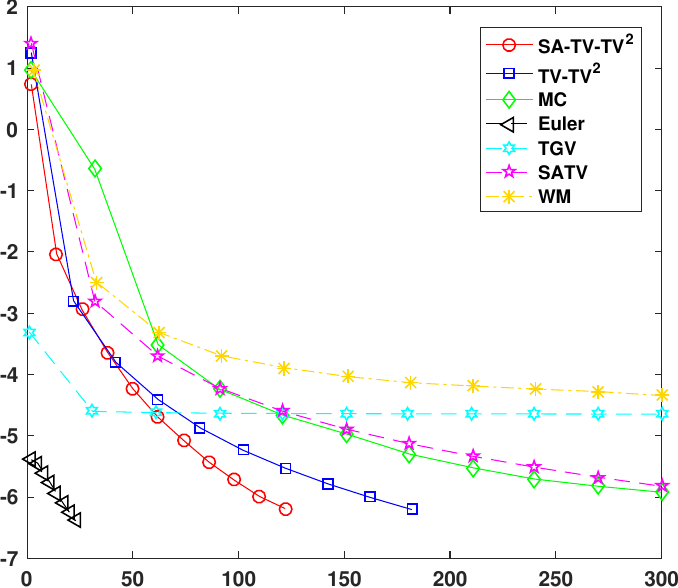}}
	\caption{Convergence curves of ``triangle" and ``peppers" by different methods.}
	\label{convergence}
\end{figure*}

\begin{figure*}[t]
      \begin{center}
			\includegraphics[width=1.00\linewidth]{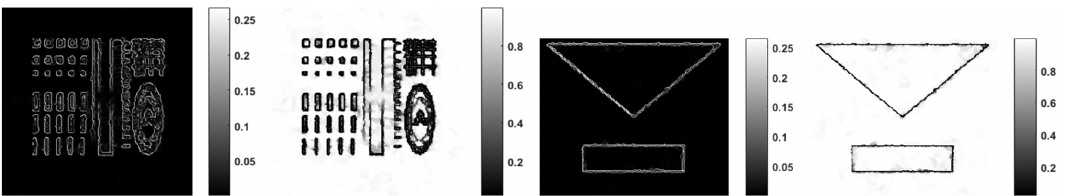}
	  \end{center}
	\caption{The spatially adaptive values of synthetic images in SA-TV-TV$^2$ method. From left to right: $\alpha(u)$ of ``bars", $\beta(u)$ of ``bars", $\alpha(u)$ of ``triangle", $\beta(u)$ of ``triangle", respectively.}
	\label{SAafabta30}
\end{figure*}

\begin{figure*}
      \begin{center}
			\includegraphics[width=1.00\linewidth]{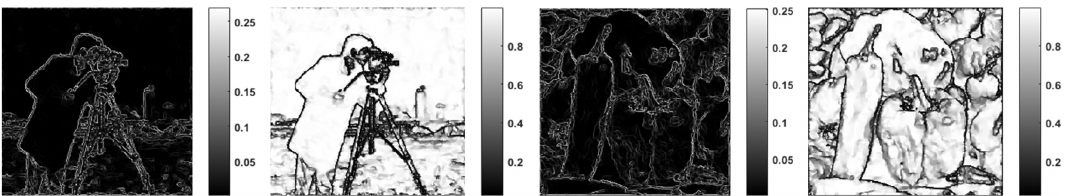}
	  \end{center}
	\caption{The spatially adaptive values of real images in SA-TV-TV$^2$ method. From left to right: $\alpha(u)$ of ``cameraman", $\beta(u)$ of ``cameraman", $\alpha(u)$ of ``peppers", $\beta(u)$ of ``peppers", respectively.}
	\label{SAafabta20}
\end{figure*}

\begin{figure*}[t]
      \begin{center}
			\includegraphics[width=1.00\linewidth]{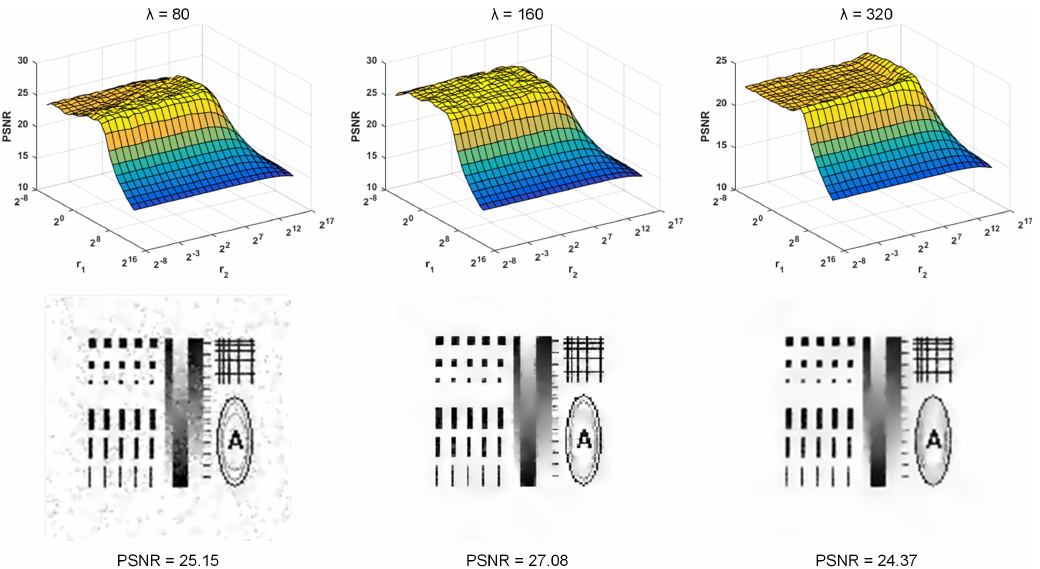}
	  \end{center}
	\caption{PSNRs of ``bars" by different penalty factors with fixed regularization parameters in SA-TV-TV$^2$ method. From left to right: the results of $\lambda=80$, $\lambda=160$ and $\lambda=320$, respectively.}
	\label{barspsnr}
\end{figure*}

We also track the decay of relative residuals \eqref{residual}, the relative errors in multipliers \eqref{multiplier}, the relative errors in $u^k$ \eqref{error} and the numerical energies \eqref{energy1}-\eqref{energy2}, which are displayed with \emph{log-scale} in Fig. \ref{cameraman20}. These plots can verify the convergence of Algorithm 1 and Algorithm 2 numerically. Fig. \ref{triangle30} records the curves of the relative error in $u^k$ and numerical energy decay of image ``triangle" and ``peppers" in \emph{log-scale} by the MC, WM and SA-TV-TV$^2$ models. Due to the dependence of gradient descent, the relative error of WM model converges much slower than the other two models, and our SA-TV-TV$^2$ model is faster and more stable than the MC model. Moreover, the numerical energies of the WM, SA-TV-TV$^2$ and MC models converge to similar values which also reveal the close relation of the three regularization terms.

Besides, we compare the CPU time consumption in Table \ref{TIME}, where the SA-TV-TV$^2$ model is much faster than other methods except for the Euler's elastica model, yet the WM model spends the highest computational cost due to the gradient descent procedure. The convergence curves in Fig. \ref{convergence} of image ``triangle" and ``peppers" also confirm that Euler's elastica and SA-TV-TV$^2$ converge faster than other approaches. Although our SA-TV-TV$^2$ model consumes more CPU time than the Euler's elastica model, it also produces much higher PSNR and SSIM values. Compared to the MC and SATV method, much CPU time is saved by our SA-TV-TV$^2$ model without any sacrifices of the recovery quality. The reason is that our SA-TV-TV$^2$ method contains fewer subproblems in each iteration and can terminate by the relative errors, while the WM, MC and SATV are all stopped by the maximum iteration number. The above evaluations convince that our Algorithm 2 can produce a similar restoration result as Algorithm 1, simultaneously saving much CPU time. Therefore, we only implement the SA-TV-TV$^2$ model in the following experiments.

It is not hard to find that the superior performance of the SA-TV-TV$^2$ model benefits from the spatially adapted regularization parameter $\alpha(u)$ and $\beta(u)$. Fig. \ref{SAafabta30} and Fig. \ref{SAafabta20} confirm that the convergent values of $\alpha(u)$ and $\beta(u)$ vary with image gradients in an opposite way. More especially, the model adaptively chooses small values of $\alpha(u)$ and large values of $\beta(u)$ in the homogeneous regions to promote the second-order regularization term for removing the noises as well as avoiding the staircase effect. On the other hand, large values of $\alpha(u)$ and small values of $\beta(u)$ are selected in textural regions to strengthen the first-order regularization term for allowing jumps and enhancing edges. In all, our SA-TV-TV$^2$ model can achieve a good trade-off between noise removal and feature preservation leading to satisfactory recovery results.

\begin{figure*}[t]
      \begin{center}
			\includegraphics[width=1.00\linewidth]{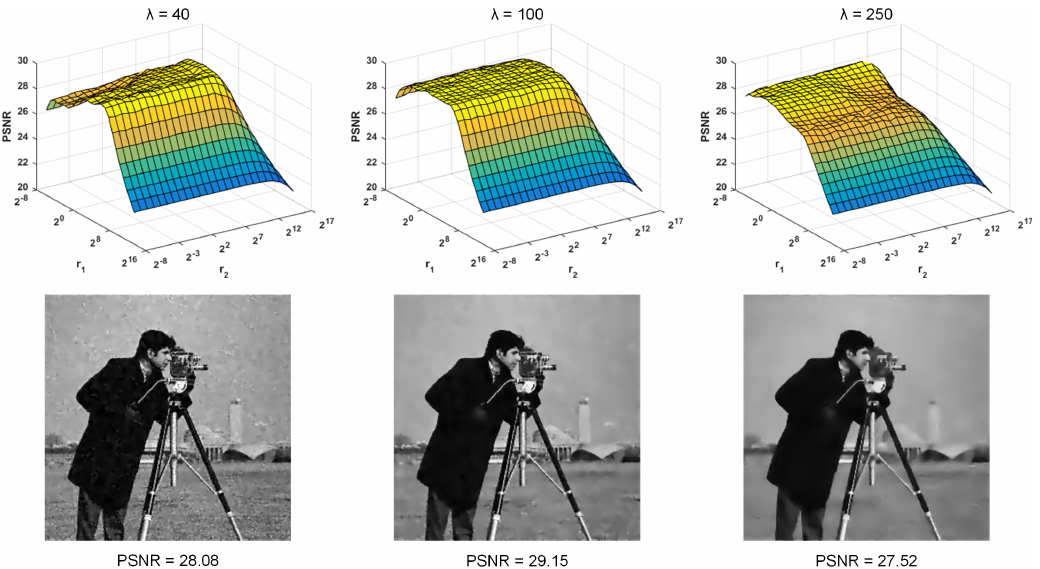}
	  \end{center}
	\caption{PSNRs of ``cameraman" by different penalty factors with fixed regularization parameters in SA-TV-TV$^2$ method. From left to right: the results of $\lambda=40$, $\lambda=100$ and $\lambda=250$, respectively.}
	\label{cameramanpsnr}
\end{figure*}

Besides, we discuss the impact of parameters $\lambda$ and $r_1$, $r_2$ in the SA-TV-TV$^2$ model on image ``bars" and ``cameraman"  to guide how to choose these parameters in practice. First, we vary the parameters $(r_1,r_2)\in\{r_1^0\times2^{-l_1},r_1^0\times2^{-l_1+1},\cdots,r_1^0\times2^{l_1-1},r_1^0\times2^{l_1}\} \times \{r_2^0\times2^{-l_2},r_2^0\times2^{-l_2+1},\cdots,r_2^0\times2^{l_2-1},r_2^0\times2^{l_2}\}$ with $r_1^0=16$, $r_2^0=32$ and $l_1=l_2=12$. Then, we select $\lambda \in \{80,160,320\}$ for the image ``bars" and $\lambda \in \{40,100,250\}$ for the image ``cameraman". As shown in Fig. \ref{barspsnr} and Fig. \ref{cameramanpsnr}, for fixed $\lambda$, there are relatively large intervals for $r_1$ and $r_2$ to generate good restoration results. Furthermore, we also show the best recovery results among various combinations of $r_1,r_2$ for each $\lambda = 80,160,320$ of the image ``bars'' in Fig. \ref{barspsnr} and $\lambda = 40,100,250$ of the image ``cameraman'' in Fig. \ref{cameramanpsnr}. It can be observed that small $\lambda$ leads to non-smoothed recovery results with some noises remaining, while large $\lambda$ results in over-smoothed recovery results with some details missing. Hence, the choice of $\lambda$ is related to the noise level of the degenerated images such that the larger the noises are, the larger $\lambda$ should be.

\begin{figure*}[t]
      \begin{center}
			\includegraphics[width=1.00\linewidth]{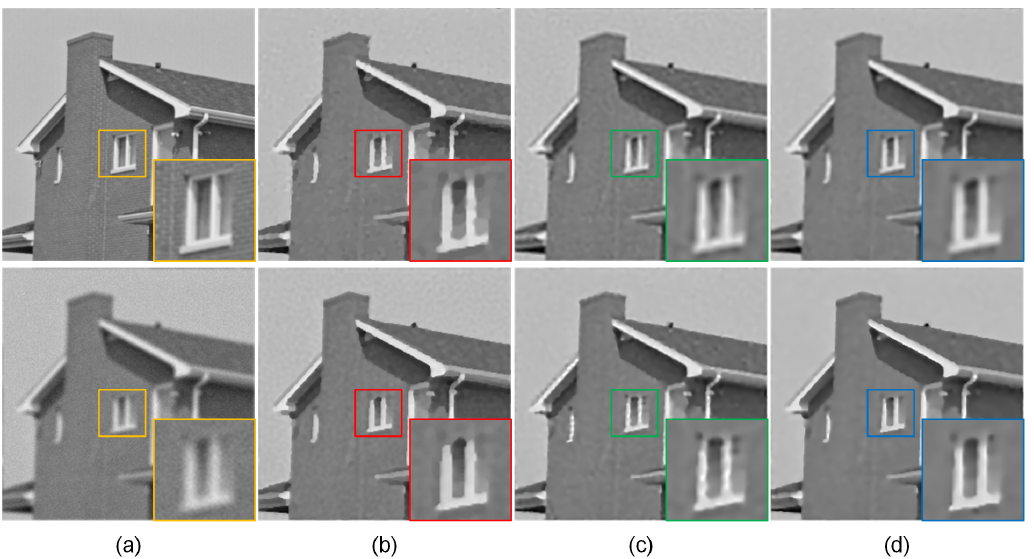}
	  \end{center}
	\caption{Deblurring comparisons of different parameters in TV-TV$^2$ and SA-TV-TV$^2$ methods on test image ``house". From left to right: (a) clear image and degraded image, (b) recovery images by TV-TV$^2$ method with $\alpha=0.4,\beta=0$ and SA-TV-TV$^2$ method with $\beta=0$, (c) recovery images by TV-TV$^2$ method with  $\alpha=0,\beta=0.4$ and SA-TV-TV$^2$ method with $\alpha=0$, (d) recovery images by TV-TV$^2$ method with $\alpha=0.4,\beta=0.4$ and SA-TV-TV$^2$ method, respectively.}
	\label{house}
\end{figure*}

\begin{table*}[t]
	\centering
	\caption{Evaluated comparisons of different parameters in TV-TV$^2$ and SA-TV-TV$^2$ methods on test image ``house" for restoring degraded image corrupted by Gaussian blur kernel with Gaussian noise of standard deviation $\sigma=5$.}
	\begin{tabular}{c|c|c|c}
		\hline\hline
		 \textbf{TV-TV$^2$} & $\alpha=0.4,\beta=0$ & $\alpha=0,\beta=0.4$ & $\alpha=0.4,\beta=0.4$  \\
        \hline
		 PSNR & $29.71$ & $29.47$ & $29.85$ \\ \hline
		 SSIM & $0.8121$ & $0.8027$ & $0.8168$  \\ \hline\hline
         \textbf{SA-TV-TV$^2$} & $\alpha=\alpha(u),\beta=0$ & $\alpha=0,\beta=\beta(u)$ & $\alpha=\alpha(u),\beta=\beta(u)$ \\ \hline
         PSNR & $29.25$ & $29.92$ & $\textbf{30.17}$ \\ \hline
         SSIM & $0.7992$ & $0.8193$ & $\textbf{0.8221}$ \\
		\hline\hline
	\end{tabular}\label{housedeblur}
\end{table*}

\subsection{Experiments on image deblurring}

In this subsection, we implement the image deblurring experiments under different degradations to illustrate the efficiency of our proposed method. The corresponding deblurring model can be formalized as follows
\begin{equation}\label{deblurringmodel}
\min_{u}~\int_{\mathrm{\Omega}}\alpha(u)|\nabla u|dx+\int_{\mathrm{\Omega}}\beta(u)|{\nabla}^2 u|_Fdx+\frac{1}{2\lambda}\int_{\mathrm{\Omega}}|Ku-f|^2dx,
\end{equation}
where the operator $K$ represents blur kernels.

The clean image ``house" is corrupted by Gaussian blur kernel (fspecial(`gaussian',[7 7],2)) and Gaussian noise of mean 0 and standard deviation 5 in Fig. \ref{house}(a), and the original image ``tomato" is degraded by the average blur kernel (fspecial(`average',[7 7])), followed by adding Gaussian noise of mean 0 with standard deviation 10 in Fig. \ref{tomato}(a). We set $r_1=r_2=4$, $\lambda=5$ and $r_1=r_2=0.2$, $\lambda=15$ for  ``house'' and ``tomato'', respectively. A series of experiments are conducted by comparing the SA-TV-TV$^2$ and TV-TV$^2$ method with different combinations of regularization parameters, i.e., $\beta=0$, $\alpha=0$ and $\alpha\neq0,\beta\neq0$. The image deblurring results and their local magnification views of the SA-TV-TV$^2$ and TV-TV$^2$ method are displayed in Fig. \ref{house} and Fig. \ref{tomato}, while the quantitative results are detailed in Table \ref{housedeblur} and \ref{tomatodeblur}.

We can see that both the recovery images of SA-TV-TV$^2$ and TV-TV$^2$ models suffer from serious staircase effect in the case of $\beta=0$, the main reason behind which is that TV regularization favors piecewise constant solutions. On the other hand, when $\alpha=0$, the results tend to be over-smoothed and with blurry edges due to the contrast reduction effect. The conclusion on visual comparisons is further confirmed by the quantitative results in terms of PSNR and SSIM as explored in Tables \ref{housedeblur} and \ref{tomatodeblur}. The best PSNR and SSIM are always obtained by the SA-TV-TV$^2$ model with non-zero spatially varying $\alpha$ and $\beta$, which demonstrate the advantages of the contrast-preserving regularization over other spatially adapted models \cite{chan2000high,bresson2007fast,Li2007,Zhang2013,Duan2016}.

\begin{figure*}[t]
      \begin{center}
			\includegraphics[width=1.00\linewidth]{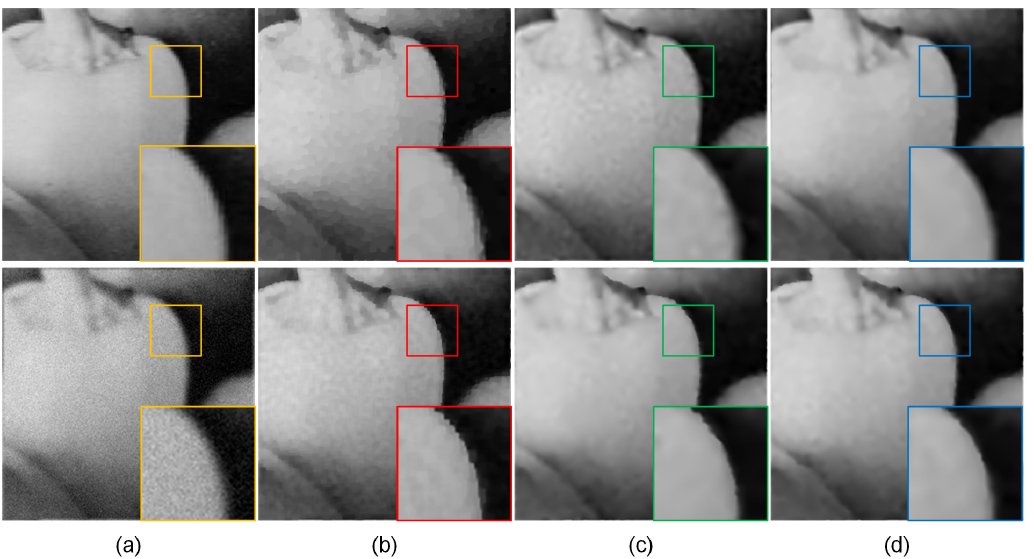}
	  \end{center}
	\caption{Deblurring comparisons of different parameters in TV-TV$^2$ and SA-TV-TV$^2$ methods on test image ``tomato". From left to right: (a) clear image and degraded image, (b) recovery images by TV-TV$^2$ method with $\alpha=1.5,\beta=0$ and SA-TV-TV$^2$ method with $\beta=0$, (c) recovery images by TV-TV$^2$ method with $\alpha=0,\beta=1.5$ and SA-TV-TV$^2$ method with $\alpha=0$, (d) recovery images by TV-TV$^2$ method with $\alpha=1.5,\beta=1.5$ and SA-TV-TV$^2$ method, respectively.}
	\label{tomato}
\end{figure*}

\begin{table*}[t]
	\centering
	\caption{Evaluated comparisons of different parameters in TV-TV$^2$ and SA-TV-TV$^2$ methods on test image ``tomato" for restoring degraded image corrupted by Average blur kernel with Gaussian noise of standard deviation $\sigma=10$.}
	\begin{tabular}{c|c|c|c}
		\hline\hline
		 \textbf{TV-TV$^2$} & $\alpha=1.5,\beta=0$ & $\alpha=0,\beta=1.5$ & $\alpha=1.5,\beta=1.5$  \\
        \hline
		 PSNR & $33.30$ & $31.82$ & $32.38$ \\ \hline
		 SSIM & $0.8810$ & $0.8885$ & $0.9050$  \\ \hline\hline
         \textbf{SA-TV-TV$^2$ }& $\alpha=\alpha(u),\beta=0$ & $\alpha=0,\beta=\beta(u)$ & $\alpha=\alpha(u),\beta=\beta(u)$ \\ \hline
         PSNR & $32.58$ & $33.75$ & $\textbf{34.24}$ \\ \hline
         SSIM & $0.8704$ & $0.9102$ & $\textbf{0.9153}$ \\
		\hline\hline
	\end{tabular}\label{tomatodeblur}
\end{table*}

\subsection{Experiments on color image denoising}
In this subsection, we implement our method on color image denoising. For the sake of simplicity, we aim to recover a color image ${\bf u}=(u^r,u^g,u^b):{\rm{\Omega}}\rightarrow\mathbb{R}^3$ channel by channel, and generate the restored image by combining the RGB channels together. Our spatially adapted first and second order regularization model for Gaussian noise removal can be defined as
\begin{equation}\label{colordenosiemodel}
\min_{\bf u}~\sum_{c \in\{r,g,b\}}(\int_{\mathrm{\Omega}}\alpha(u^c)|\nabla u^c|dx+\int_{\mathrm{\Omega}}\beta(u^c)|{\nabla}^2 u^c|_Fdx)+\sum_{c \in\{r,g,b\}} \frac{1}{2\lambda}\int_{\mathrm{\Omega}}|u^c-f^c|^2dx.
\end{equation}
As shown in Fig. \ref{colorDenoising}, the color images ``lena" and ``flower" are degraded by Gaussian noise with mean zero and the standard deviation $\sigma=\{20,30\}$, respectively. We use $r_1=1$, $r_2=2$ and $\lambda=\{80,160\}$ for the two images accordingly.  Although all methods can remove the noises and recover main image structural information, our method gives the best visual quality with not only sharp and clear edges but also the homogeneity in slanted regions. The corresponding qualitative evaluations are provided in Table \ref{colordenoise}, which also convince the sound effects of our proposal on color image denoising.

\begin{figure*}[t]
      \centering
      \subfigure{
			\includegraphics[width=0.19\linewidth]{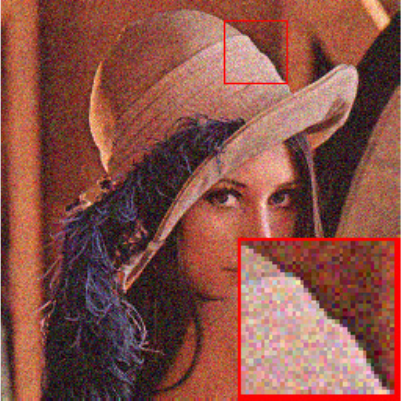}}\hspace{-1ex}
      \subfigure{
            \includegraphics[width=0.19\linewidth]{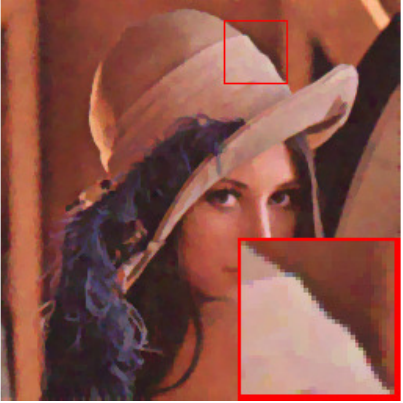}}\hspace{-1ex}
      \subfigure{
			\includegraphics[width=0.19\linewidth]{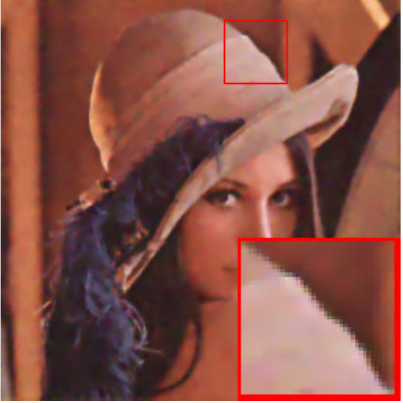}}\hspace{-1ex}
      \subfigure{
			\includegraphics[width=0.19\linewidth]{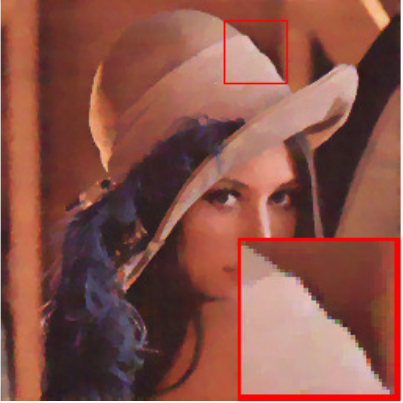}}\hspace{-1ex}
      \subfigure{
			\includegraphics[width=0.19\linewidth]{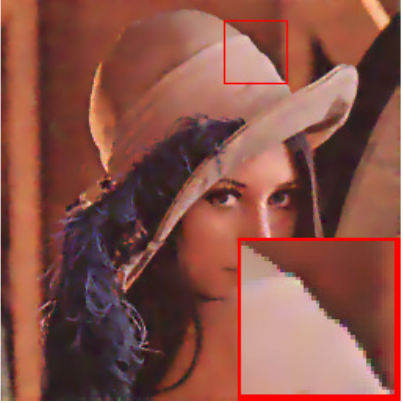}}\\
      \setcounter{subfigure}{0}
      \vspace{-0.25cm}
      \subfigure[Noisy images]{
            \includegraphics[width=0.19\linewidth]{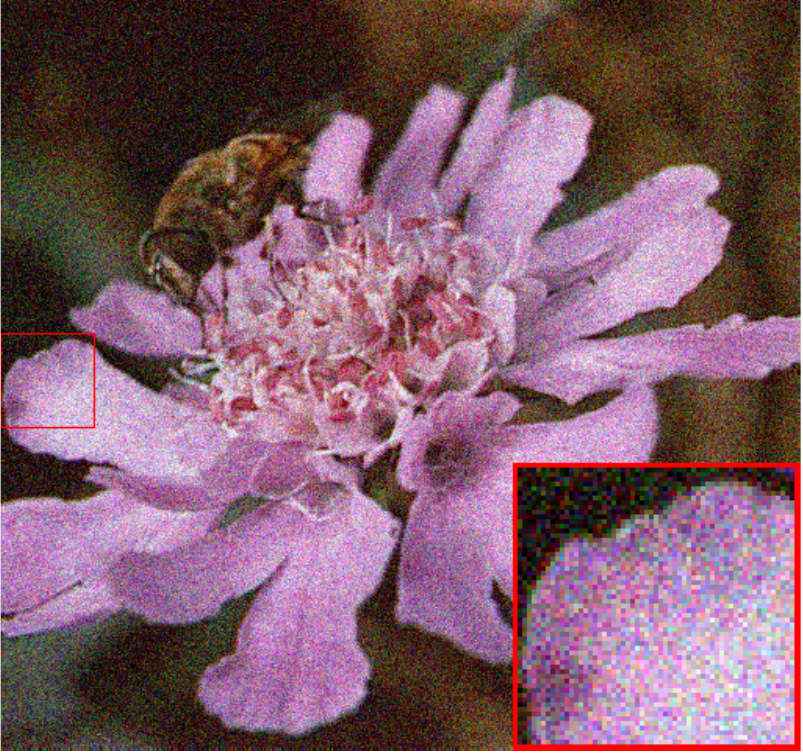}}\hspace{-1ex}
      \subfigure[Euler]{
            \includegraphics[width=0.19\linewidth]{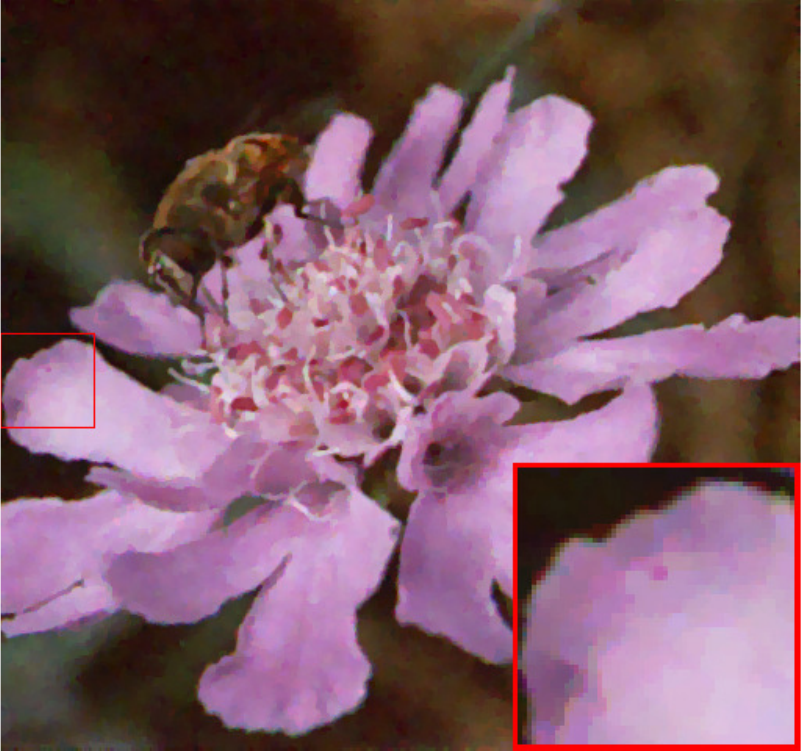}}\hspace{-1ex}
      \subfigure[TV-TV$^2$]{
            \includegraphics[width=0.19\linewidth]{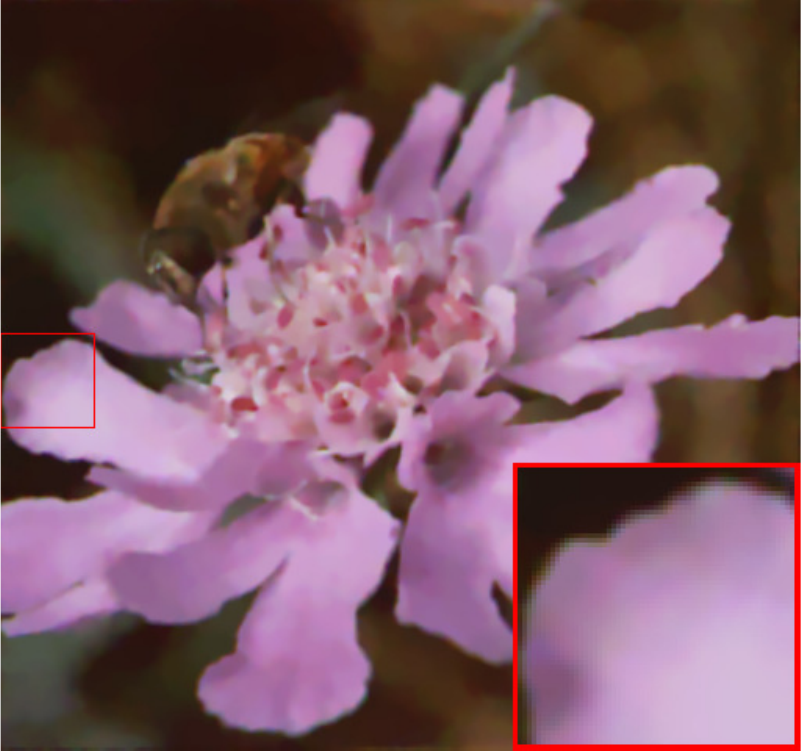}}\hspace{-1ex}
      \subfigure[TGV]{
            \includegraphics[width=0.19\linewidth]{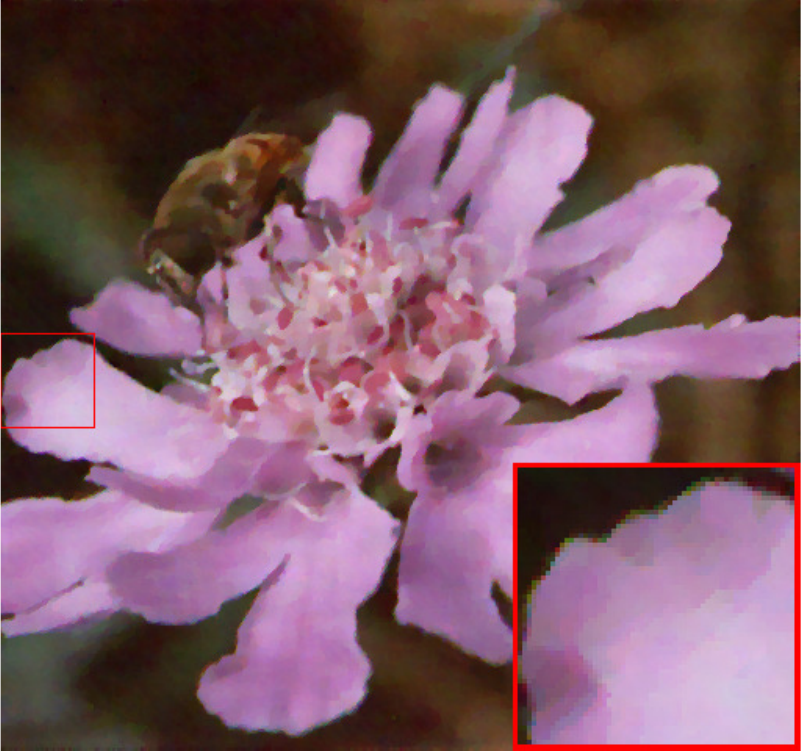}}\hspace{-1ex}
      \subfigure[SA-TV-TV$^2$]{
			\includegraphics[width=0.19\linewidth]{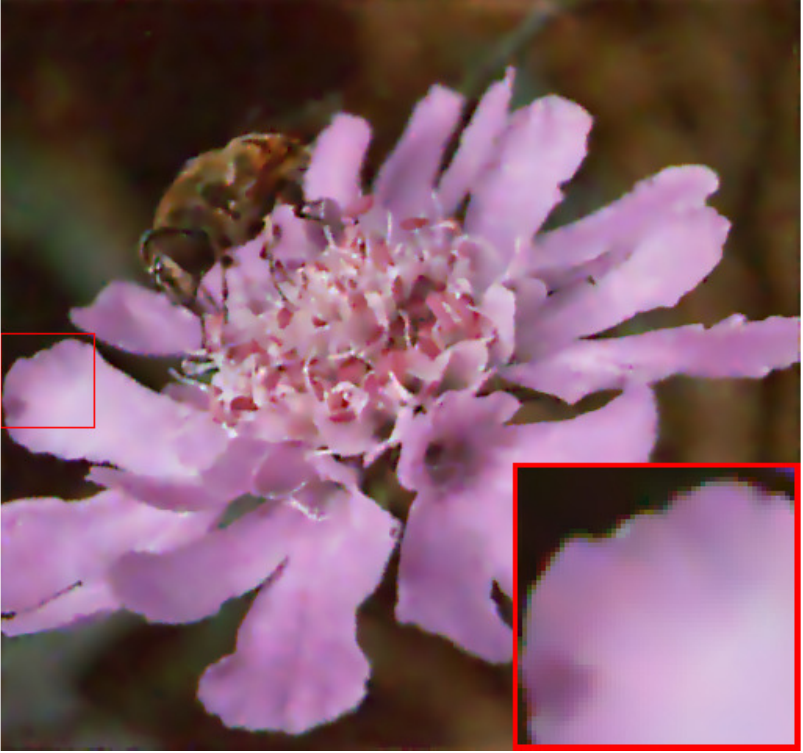}}
	\caption{The color image denoising results of ``lena" ($\sigma=20$) (top) and ``flower" ($\sigma=30$) (bottom) generated by the comparative methods. From left to right: (a) noisy images; (b) recovery images by Euler's elastica; (c) recovery images by TV-TV$^2$; (d) recovery images by TGV; (e) recovery images by SA-TV-TV$^2$.}
	\label{colorDenoising}
\end{figure*}

\begin{table*}[t]
\caption{Comparison of PSNR and SSIM on color image denoising examples among the Euler's elastica, TV-TV$^2$, TGV and SA-TV-TV$^2$ methods.}
\label{colordenoise}
\begin{center}
\begin{tabular}{c|c|c|c|c|c|c|c|c}
\hline\hline
Methods & \multicolumn{2}{|c}{Euler's elastica} & \multicolumn{2}{|c}{TV-TV$^2$} & \multicolumn{2}{|c}{TGV} & \multicolumn{2}{|c}{SA-TV-TV$^2$}  \\
\hline
Images & PSNR & SSIM & PSNR & SSIM & PSNR & SSIM & PSNR & SSIM \\
\hline\hline
lena $(\sigma=20)$ & 29.32 & 0.9406 & 29.25 & 0.9392 & 29.15 & 0.9363 & \bf{29.88} & \bf{0.9453} \\
\hline
flower $(\sigma=30)$ & 30.26 & 0.9478 & 30.38 & 0.9520 & 29.83 & 0.9471 & \bf{30.98} & \bf{0.9575} \\
\hline\hline
\end{tabular}
\end{center}
\end{table*}

\subsection{Experiments on image inpainting}

Finally, we demonstrate some examples of our SA-TV-TV$^2$ method on image inpainting problems. In general, the task of image inpainting is to reconstruct a missing part of an image using information from the intact part. The missing part of the image is called the inpainting domain and is denoted by $D \subseteq \mathrm{\Omega}$. Image inpainting has been extensively studied in the literature including TV inpainting \cite{getreuer2012total}, curvature driven diffusion inpainting \cite{chan2001non}, Mumford-Shah based inpainting \cite{esedoglu2002digital} and Euler's elastica inpainting \cite{tai2011fast}. The spatially varying first and second order regularization inpainting model is described as follows
\[\min_{u}~\int_{\mathrm{\Omega}}\alpha(u)|\nabla u|dx+\int_{\mathrm{\Omega}}\beta(u)|{\nabla}^2 u|_Fdx+\frac{1}{2\lambda}\int_{\mathrm{\Omega\setminus D}}(u-f)^2dx.\]
In order to obtained an efficient ADMM algorithm, we introduce three auxiliary variables and rewrite the above minimization problem into the following constrained one
\begin{equation}\label{inpaintingmodel}
\begin{split}
& \min_{u,v,w}~\int_{\mathrm{\Omega}}\alpha(u)|v|dx+\int_{\mathrm{\Omega}}\beta(u)|w|_Fdx+\frac{1}{2\lambda}\int_{\mathrm{\Omega\setminus D}}(z-f)^2dx \\
& ~~\mathrm{s.t.,}~~~z=u,~v=\nabla u,~ w={\nabla}^2 u.
\end{split}
\end{equation}
More details for dealing with the constrained optimization problem \eqref{inpaintingmodel} can be referred to \cite{papafitsoros2013combined}.

In Fig. \ref{Inpaintingimages}, we present two convincing examples of image inpainting by our method. We can observe that the reconstructed regions can naturally blend into background, see Fig. \ref{Inpaintingimages} (a2) and (b2). In addition, we compare the SA-TV-TV$^2$ model, TV-TV$^2$ model \cite{papafitsoros2013combined} and Euler's elastica model \cite{tai2011fast} on a simple synthetic image. As shown by Fig. \ref{inpaintingcompare} (b) and (f), the TV inpainting model with the constant regularization parameter gives nearly piecewise constant result inside the inpainting domain, while the TV model with adaptive parameter also fails to fill such a large gap in between the two branches. Actually, the $\mathrm{TV}^2$ model can somehow connect the gap as shown in Fig. \ref{inpaintingcompare} (c) and (g) with the price of some blur. Similar problem happened to Fig. \ref{inpaintingcompare} (d) and (e), which are obtained by the TV-TV$^2$ method and the Euler's elastica method, respectively. It is clearly shown that our SA-TV-TV$^2$ model gives the visually best inpainting result, which can recover the gap using straight edges; see Fig. \ref{inpaintingcompare} (h).

\begin{figure*}[t]
      \begin{center}
			\includegraphics[width=1.00\linewidth]{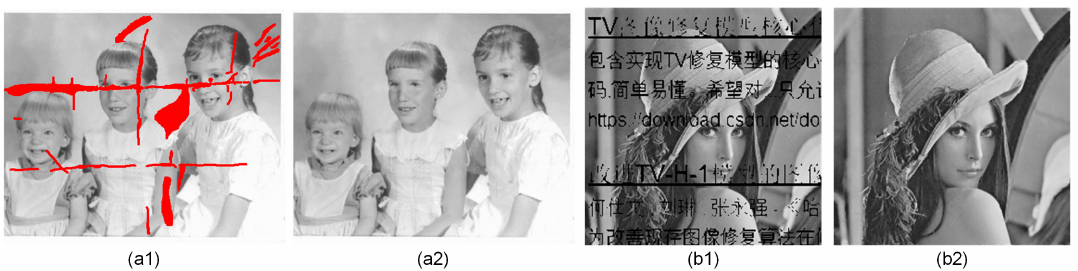}
	  \end{center}
	\caption{Inpainting results of real images by the SA-TV-TV$^2$ method. The parameters are selected as $r_1=2$, $r_2=4$, $r_3=0.005$ and $\lambda=2$.}
	\label{Inpaintingimages}
\end{figure*}

\begin{figure*}[t]
      \begin{center}
			\includegraphics[width=1.00\linewidth]{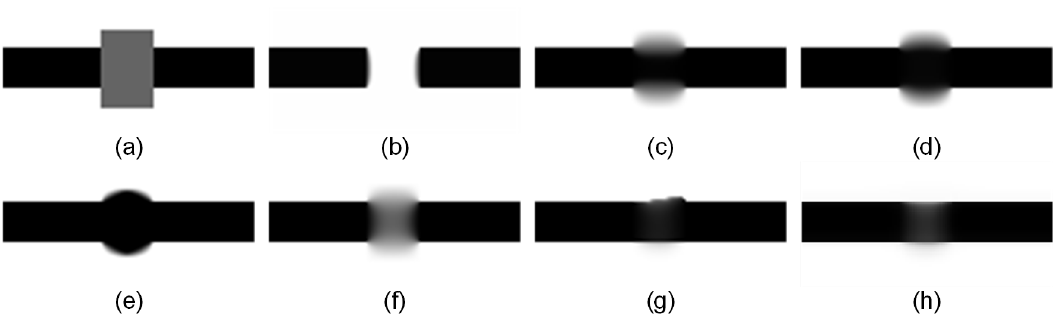}
	  \end{center}
	\caption{Inpainting comparisons of different parameters in TV-TV$^2$, Euler's elastica and SA-TV-TV$^2$ methods. From left to right and top to bottom: (a) the degraded image, (b) TV with $\alpha=10$, (c) $\mathrm{TV}^2$ with $\beta=5$, (d) TV-TV$^2$ with $\alpha=10,\beta=5$, (e) Euler's elastica model, (f) TV with $\alpha(u)$ in \eqref{wtv2}, (g) $\mathrm{TV}^2$ with $\beta(u)$ in \eqref{wtv2} and (h) SA-TV-TV$^2$.}
	\label{inpaintingcompare}
\end{figure*}

\section{Conclusion}
\label{sec6}
In this paper, we proposed a novel Weingarten map minimization model for image restoration problems. Our model was shown can ideally preserve image contrast, edges and corners of objects. We developed an ADMM-based algorithm for solving the high order variational model. More than that, we further derived a spatially adapted first and second order regularization from the Weingarten map, and presented a more efficient algorithm by solving all subproblems with either FFT or closed-form solution. Numerous numerical experiments are conducted on both synthetic and real images to demonstrate the efficacious and ascendant performance of the proposed models. By comparing with other well established high order models, we showed the spatially adapted first and second order regularization can not only keep image intensity contrast and eliminate staircase effect, but also save the computational cost.

\section*{Appendix A}
Let $f=h\chi_A(x,y)$ be a binary function defined on a rectangle domain $\rm\Omega = (-2R, 2R)\times (-2R, 2R)$ with $A\subset\mathrm{\Omega}$ being an arbitrary open set with $C^2$ boundary. We consider the following integral of Weingarten map regularization by triangle inequality
\begin{align*}
\int_{\mathrm{\Omega}} |W_{f}|_Fdxdy &= \int_{\mathrm{\Omega}} \bigg|\nabla \bigg(\frac{\nabla f}{\sqrt{1+|\nabla f|^2}}\bigg)\bigg|_Fdxdy \\
&= \int_{\mathrm{\Omega}}\sqrt{{\bigg|\nabla\frac{f_x}{\sqrt{1+|\nabla f|^2}}\bigg|}^2+{\bigg|\nabla\frac{f_y}{\sqrt{1+|\nabla f|^2}}\bigg|}^2}dxdy \\
& \leq \int_{\mathrm{\Omega}}\bigg|\nabla\frac{f_x}{\sqrt{1+|\nabla f|^2}}\bigg|dxdy+\int_{\mathrm{\Omega}}\bigg|\nabla\frac{f_y}{\sqrt{1+|\nabla f|^2}}\bigg|dxdy,
\end{align*}
where $f_x$ and $f_y$ are the first-order differential operators of $f$, i.e., $\nabla f=(f_x,f_y)$, and $|\cdot|$ denotes the Euclidean norm. Similar to the total variation regularization, owing to $\frac{f_x}{\sqrt{1+|\nabla f|^2}}\leq1$ and $\frac{f_y}{\sqrt{1+|\nabla f|^2}}\leq1$ at every point $(x,y)$ on $\mathrm{\Omega}$, we can arrive at the following conclusion based on divergence theorem and Cauchy-Schwartz inequality
\begin{align*}
\int_{\mathrm{\Omega}} |W_{f}|_Fdxdy &\leq \int_{\mathrm{\Omega}}\bigg|\nabla\frac{f_x}{\sqrt{1+|\nabla f|^2}}\bigg|dxdy+\int_{\mathrm{\Omega}}\bigg|\nabla\frac{f_y}{\sqrt{1+|\nabla f|^2}}\bigg|dxdy \\
&= \sup_{\substack{p\in{C_c^1}(\mathrm{\Omega},\mathbb{R}^{n}) \\ \|p\|_\infty\leq 1}}\int_{\rm\Omega}\frac{f_x}{\sqrt{1+|\nabla f|^2}}\div pdxdy +\sup_{\substack{q\in{C_c^1}(\mathrm{\Omega},\mathbb{R}^{n}) \\ \|q\|_\infty\leq 1}}\int_{\rm\Omega}\frac{f_y}{\sqrt{1+|\nabla f|^2}}\div qdxdy \\
&= \sup_{\substack{p\in{C_c^1}(\mathrm{\Omega},\mathbb{R}^{n}) \\ \|p\|_\infty\leq 1}}\int_{\partial A}\frac{f_x}{\sqrt{1+|\nabla f|^2}} p\cdot\nu d\mathcal{H}^1 +\sup_{\substack{q\in{C_c^1}(\mathrm{\Omega},\mathbb{R}^{n}) \\ \|q\|_\infty\leq 1}}\int_{\partial A}\frac{f_y}{\sqrt{1+|\nabla f|^2}} q\cdot\nu d\mathcal{H}^1\\
&\leq \sup_{\substack{p\in{C_c^1}(\mathrm{\Omega},\mathbb{R}^{n}) \\ \|p\|_\infty\leq 1}}\int_{\partial A}\bigg|\frac{f_x}{\sqrt{1+|\nabla f|^2}}\bigg| |p\cdot\nu| d\mathcal{H}^1+ \sup_{\substack{q\in{C_c^1}(\mathrm{\Omega},\mathbb{R}^{n}) \\ \|q\|_\infty\leq 1}}\int_{\partial A}\bigg|\frac{f_y}{\sqrt{1+|\nabla f|^2}}\bigg| |q\cdot\nu| d\mathcal{H}^1 \\
&\leq 2\int_{\partial A}d\mathcal{H}^1 = 2Per(A,\mathrm{\Omega}),
\end{align*}
which shows that the integral of Weingarten map is independent of $h$.

\section*{Appendix B}
Let $f=\sum_{i=1}^n h_i\chi_i(x,y)$ be a piecewise constant function defined on a rectangle domain ${\mathrm{\Omega}}$, where $\chi_i$ is the characteristic function of the subdomain ${\rm\Omega}_i$.
Similarly, we can reformulate the Weingarten map regularization over the image domain $\rm\Omega$ as follows
\begin{align*}
\int_{\mathrm{\Omega}} |W_{f}|_Fdxdy & \leq \sum_{i=1}^n\int_{\mathrm{\Omega}}\bigg|\nabla\frac{{h_i\chi_i}_x}{\sqrt{1+|h_i\nabla \chi_i|^2}}\bigg|dxdy +\sum_{i=1}^n\int_{\mathrm{\Omega}}\bigg|\nabla\frac{{h_i\chi_i}_y}{\sqrt{1+|h_i\nabla {\chi_i}|^2}}\bigg|dxdy \\
&=\sum_{i=1}^n \sup_{\substack{p_i\in{C_c^1}(\mathrm{\Omega},\mathbb{R}^{n}) \\ \|p_i\|_\infty\leq 1}}\int_{\rm\Omega}\frac{{h_i\chi_i}_x}{\sqrt{1+|h_i\nabla {\chi_i}|^2}}\div p_idxdy +\sum_{i=1}^n \sup_{\substack{q_i\in{C_c^1}(\mathrm{\Omega},\mathbb{R}^{n}) \\ \|q_i\|_\infty\leq 1}}\int_{\rm\Omega}\frac{{h_i\chi_i}_y}{\sqrt{1+|h_i\nabla {\chi_i}|^2}}\div q_idxdy \\
&= \sum_{i=1}^n \sup_{\substack{p_i\in{C_c^1}(\mathrm{\Omega},\mathbb{R}^{n}) \\ \|p_i\|_\infty\leq 1}}\int_{\partial \mathrm{\Omega}_i}\frac{{h_i\chi_i}_x}{\sqrt{1+|h_i\nabla {\chi_i}|^2}} p_i\cdot\nu_i d\mathcal{H}^1 +\sum_{i=1}^n \sup_{\substack{q_i\in{C_c^1}(\mathrm{\Omega},\mathbb{R}^{n}) \\ \|q_i\|_\infty\leq 1}}\int_{\partial \mathrm{\Omega}_i}\frac{{h_i\chi_i}_y}{\sqrt{1+|h_i\nabla {\chi_i}|^2}} q_i\cdot\nu_i d\mathcal{H}^1\\
&\leq \sum_{i=1}^n \sup_{\substack{p_i\in{C_c^1}(\mathrm{\Omega},\mathbb{R}^{n}) \\ \|p_i\|_\infty\leq 1}}\int_{\partial \mathrm{\Omega}_i}\bigg|\frac{{h_i\chi_i}_x}{\sqrt{1+|h_i\nabla {\chi_i}|^2}}\bigg| |p_i\cdot\nu_i| d\mathcal{H}^1+ \sum_{i=1}^n \sup_{\substack{q_i\in{C_c^1}(\mathrm{\Omega},\mathbb{R}^{n}) \\ \|q_i\|_\infty\leq 1}}\int_{\partial \mathrm{\Omega}_i}\bigg|\frac{{h_i\chi_i}_y}{\sqrt{1+|h_i\nabla {\chi_i}|^2}}\bigg| |q_i\cdot\nu_i| d\mathcal{H}^1 \\
&\leq \sum_{i=1}^n 2\int_{\partial \mathrm{\Omega}_i}d\mathcal{H}^1 = 2\sum_{i=1}^n Per(\mathrm{\Omega}_i,\mathrm{\Omega}),
\end{align*}
which is also independent of $h$.

\begin{acknowledgements}
The authors would like to thank Prof. Wei Zhu from the University of Alabama for providing us with MATLAB code of \cite{zhu2013augmented}. We also would like to thank the anonymous referees and the editor for the valuable comments and helpful suggestions to improve this paper. The work is supported by National Natural Science Foundation of China (NSFC 12071345, 11701418), Major Science and Technology Project of Tianjin 18ZXRHSY00160 and Recruitment Program of Global Young Expert.
\end{acknowledgements}



\end{document}